\documentclass[a4paper]{article}

\usepackage{amsmath,amsthm}
\usepackage{amssymb}

\usepackage{enumitem}

\newtheorem{theorem}{Theorem}[section]
\newtheorem{corollary}[theorem]{Corollary}
\newtheorem{lemma}[theorem]{Lemma}

\theoremstyle{definition}

\newtheorem{remark}[theorem]{Remark}

\numberwithin{equation}{section}

\newtheorem*{lemma*}{Lemma}
\newtheorem*{theorem*}{Theorem}

\usepackage{sagetex}

\usepackage{dsfont}

\DeclareMathAlphabet{\mathpzc}{OT1}{pzc}{m}{it}

\newenvironment{mysage}{\sagesilent}{\endsagesilent}

\begin{document}

\begin{sagesilent}
################
### Rounding ####
###############

def roundup(x,d):

    return float(ceil(x*10^d)/10^d)
def rounddown(x,d):
    return float(floor(x*10^d)/10^d)

##################################
### Managing Number Representation ###
##################################

def Interval(x): ### Converts a real x into a neighborhood of x, expressed by a closed interval. Allows us to handle interval arithmetic. ###
    X=RBF(RIF(x)) 
    I=RealSet([X.lower(),X.upper()])[0]
    return I

def Upper(x,d): ### Gives a ''close'' upper bound with d decimal digits for a real x. ###
    I=Interval(x)
    u=float(ceil(I.upper()*10^d)/10^d)
    return u

def Lower(x,d): ### Gives a ''close'' lower bound with d decimal digits for a real x. ###
    I=Interval(x)
    l=float(floor(I.lower()*10^d)/10^d)
    return l

def Trunc(x,d): ### Given a real x, it gives an expression of x with at most d EXACT decimals; handy for managing displaying of numbers without losing accuracy (upon inserting ''...''). It is a truncation that manages to take care of short decimals as well as numbers of the form 2.9999999999999 (that are taken to be 3). It handles other numbers normally. ###

    if d>10:
        return false
    I=Interval(x)
    s=0
    while (I.lower()*10^s).trunc()==(I.upper()*10^s).trunc():
        s=s+1
    if (x-RIF(x))!=0: ### Ex. x=29.99999999999999999 => RIF(x)=30. This operation returns 29.99999 (d=5), rather than 30 ###
        return Lower(x,d)
    if d>s and (x-RIF(x))==0: ### Ex. x=1.24 and d=5. This operation returns 1.24 rather than 1.23999 given by above orders. ###
        return float(x)
    return float(floor(I.lower()*10^d)/10^d) 

def Numb(x,y,d):
    z=RIF((x+y)/2)
    return Trunc(z,d)

gamma = 0.5772156649015328606065120900824024310421

def A(a,d):
    if a==1:
        if 0<d<=1:
            return RIF(max(gamma,1/d/exp(gamma*d+1)))
        return False
    if a!=1:
        if d+1==a:
            return RIF(max(1/2,1/abs(a-1), zeta(a)-1/(a-1)))
        if 0<d<a<d+1:
            l1=RIF((d-a+1)/abs(zeta(a))/abs(a-1))
            return RIF(max(1/2, (l1^(d-a+1)/d^d)^(1/(a-1)), zeta(a)-1/(a-1)))             
        if d==a:
            return RIF(1/2)
        return False

def crux(v):
    if v%2==0:
        return CONSTANT2
    return CONSTANT_RAM
	
#--------------------

def fp_prime(p): #Definition solely on the primes -- TO BE CHANGED EACH TIME
    f=RIF((p-2)/(p^(3/2)-p-sqrt(p)+2))
    return RIF(1+f)

def pp(r): #q square free
    f=1
    p=2
    P=1
    if moebius(r)^2==1:
        while p<=sqrt(r):
            if is_prime(p)==True:
                if r%p==0:
                    f=f*fp_prime(p) 
                    P=P*p
            p=p+1
        if P!=r:
            j=r/P
            f=f*fp_prime(j)
        return RIF(f)
    return 0 

def ramp_prime(p): #Definition solely on the primes -- TO BE CHANGED EACH TIME
    a=p^(3/2)+p-sqrt(p)-1
    b=p^(3/2)-sqrt(p)+1
    f=RIF(a/b)
    if p==2:
        return RIF(21/25)
    return RIF(f)

def ramp(r): #q square free
    f=1
    p=2
    P=1
    if moebius(r)^2==1:
        while p<=sqrt(r):
            if is_prime(p)==True:
                if r%p==0:
                    f=f*ramp_prime(p) 
                    P=P*p
            p=p+1
        if P!=r:
            j=r/P
            f=f*ramp_prime(j)
        return RIF(f)
    return 0  

#--------------------

def E(a,v):
    if a!=1:
        a1= RIF(abs(a-1)/(a-1/2))
        A = RIF(Upper(Ram_new_cst,2)*(1+a1))
        if v==2:
            A = RIF(Upper((1-1/sqrt(2))*max(C2_v2_1,C2_v2_2),digits)*(1+a1))
        B = RIF(abs(zeta(a)/zeta(2*a)-6/pi^2/(a-1)))
        if v==2:
            B = RIF((1-1/sqrt(2))*abs(v^a/(v^a+1)*zeta(a)/zeta(2*a)-v/(v+1)*6/pi^2/(a-1)))
        c1=RIF(abs(zeta(a)*(a-1)))
        C=RIF(a1*( 3*zeta(2*a)/(a-1/2)/pi^2/c1 )^(2/(a-1))) 
        if v==2:
            C=RIF((1-1/sqrt(2))*a1*(3*(v^a+1)*zeta(2*a)/v^(a-1)/(v+1)/(a-1/2)/pi^2/c1)^(2/(a-1))) 
        return max(A,B,C)
    return -1

\end{sagesilent}

\begin{mysage} 

#For checks about the following constants, see at the end of the file.

digits = 3 #Precision seems to be affected by more than 7 digits

dlong = 5 # (Trunc(,) is not defined for higher number of digits)

delta = 1/3
delta2 = 45/32
deltaram = 1/3

choice = 3/4

theta = RIF(1-1/12/log(10))

exponent=6
exponent2=8
Threshold=10^exponent
Threshold2=10^(exponent/2)
Threshold3=5*10^(exponent2)

Thresholdv2=10^exponent2

expbip= 7 
Thresholdbip = 10^(expbip)
Thresholdbip2 = 10^(expbip/2)

##############################################################################################

# ----> lemma-label{tildetildem}

I_sum1_l=0.755366607315099
I_sum1_u=0.755366626776258	### Precision: 10^9, Time: 8555.87245297s, using SAGE ###
C_sum1=Numb(I_sum1_l+gamma, I_sum1_u+gamma,8)

k2_1half = RIF ( 3/(3-3^(1-1/2)) )
k2_theta = RIF ( 3/(3-3^(1-theta)) )

Delta_1half_l=Lower(zeta(3/2)*1.36843276585094,dlong)	### Precision: 4*10^9, using C++ ###
Delta_1half_u=Upper(zeta(3/2)*1.36843284087041,dlong)	
Delta_1half=Trunc(Delta_1half_u,digits)

Delta_theta_l=Lower(zeta(1+theta)*1.00724550163589,dlong)
Delta_theta_u=Upper(zeta(1+theta)*1.00724557626645,dlong)		
Delta_theta=Trunc(Delta_theta_u,digits)

##############################################################################################

# ----> lemma-label{sum1}

I_Error_sum1_l=7.35984783249704
I_Error_sum1_u=7.35984795957735	### Precision 10^9, Time: 53510.6989071s, using SAGE ###
Error_sum1=Upper(I_Error_sum1_u,digits)

#DELTA PRODUCTS

Error_sum1_1= RIF(A(1,1/30) * zeta(2 - 2/30) * zeta(2-1/30) * 1.20606653924993)			#d=1/30, Pfd=[1.20606644302573,1.20606653924993]
Error_sum1_2 = RIF(A(1,1/9) * zeta(2 - 2/9) * zeta(2-1/9) * 1.1861090282627)				#d=1/9, Pfd=[1.18610893203787,1.1861090282627]
Error_sum1_3 = RIF(A(1,1/7) * zeta(2 - 2/7) * zeta(2-1/7) * 1.17508136980136)			#d=1/7, Pfd=[1.1750812735761,1.17508136980136]
Error_sum1_4 = RIF(A(1,1/5) * zeta(2 - 2/5) * zeta(2-1/5) * 1.15005486834993)			#d=1/5, Pfd=1.15005477212383,1.15005486834993]
Error_sum1_5 = RIF(A(1,1/3) * zeta(2 - 2/3) * zeta(2-1/3) * 1.0571339931299)			#d=1/3, Pfd=[1.05713389690357,1.0571339931299]
Error_sum1_6 = RIF(A(1,2/5) * zeta(2 - 2*2/5) * zeta(2-2/5) * 0.985530751288352)			#d=2/5/14, Pfd=[0.985530703014949,0.985530751288352]
Error_sum1_7 = RIF(A(1,4/9) * zeta(2 - 2*4/9) * zeta(2-4/9) * 0.925136718595616)			#d=4/9, Pfd=[0.925136670211812,0.925136718595616]
Error_sum1_8 = RIF(A(1,8/17) * zeta(2 - 2*8/17) * zeta(2-8/17) * 0.883955928323831)		#d=8/17, Pfd=[0.883955879867484,0.883955928323831]

##############################################################################################

# ----> lemma-label{sumvar1log}

Prod_sumvar1log_l = 0.908260084203484  ### Precision: 3*10^9, Time: <360s, using C++ ---> int_double
Prod_sumvar1log_u = 0.908260346651785

deltaProd_sumvar1log_l=8.06663259802707
deltaProd_sumvar1log_u=8.08903653041877

Sum_sumvar1log_l = 0.550970060568342+gamma
Sum_sumvar1log_u = 0.55099105871815+gamma

A2_sumvar1log = RIF(1+(2-2^delta-2)/((2-1)*2^(1-delta)+2^delta+1))

p2_sumvar1log = RIF(1-A2_sumvar1log/(2-1+A2_sumvar1log))
q2_sumvar1log = RIF(log(2)*A2_sumvar1log/(A2_sumvar1log+2-1))
v2_sumvar1log = RIF(1+(2*(2-1)-A2_sumvar1log*(2+2^delta))/((2-1)*2^(1-delta)+A2_sumvar1log*(2+2^delta)-2+1))

constant_g10 = RIF( Prod_sumvar1log_u/2 + 1/log(Threshold2) * (Sum_sumvar1log_u ) + A(1,delta)*deltaProd_sumvar1log_u/delta/log(Threshold2)^2 )
constant_g20 = RIF( p2_sumvar1log * Prod_sumvar1log_u/2 + 1/log(Threshold2) * (q2_sumvar1log + Sum_sumvar1log_u) + v2_sumvar1log * A(1,delta)*deltaProd_sumvar1log_u/delta/log(Threshold2)^2 )

constant_g10BIP = RIF( Prod_sumvar1log_u/2 + 1/log(Thresholdbip2) * (Sum_sumvar1log_u ) + A(1,delta)*deltaProd_sumvar1log_u/delta/log(Thresholdbip2)^2 )
constant_g20BIP = RIF( p2_sumvar1log * Prod_sumvar1log_u/2 + 1/log(Thresholdbip2) * (q2_sumvar1log + Sum_sumvar1log_u) + v2_sumvar1log * A(1,delta)*deltaProd_sumvar1log_u/delta/log(Thresholdbip2)^2 )

##############################################################################################

# ----> lemma-label{sum2}

I_sum2_l=1.13992197915589		### Precision: 10^9, Time: 8555.87245297s, using SAGE ###
I_sum2_u=1.13992201807822		
C_sum2=Numb(I_sum2_l+gamma, I_sum2_u+gamma,8)
C_sum2_22=Numb(I_sum2_l+gamma+log(2)/3, I_sum2_u+gamma+log(2)/3,8)

Error_sum2=RIF( A(1,delta)*zeta(2-2*delta)/zeta(4-4*delta) ) 

#DELTA PRODUCTS

Error_sum1_1= RIF(A(1,1/30) * zeta(2 - 2/30) / zeta(4-4/30) )			#d=1/30
Error_sum1_2 = RIF(A(1,1/9) * zeta(2 - 2/9) / zeta(4-4/9) )			#d=1/9
Error_sum1_3 = RIF(A(1,1/7) * zeta(2 - 2/7) / zeta(4-4/7) )			#d=1/7
Error_sum1_4 = RIF(A(1,1/5) * zeta(2 - 2/5) / zeta(4-4/5) )			#d=1/5
Error_sum1_5 = RIF(A(1,1/3) * zeta(2 - 2/3) / zeta(4-4/3) )			#d=1/3
Error_sum1_6 = RIF(A(1,2/5) * zeta(2 - 2*2/5) / zeta(4-2*4/5) )			#d=2/5
Error_sum1_7 = RIF(A(1,4/9) * zeta(2 - 2*4/9) / zeta(4-4*4/9) )			#d=4/9
Error_sum1_8 = RIF(A(1,8/17) * zeta(2 - 2*8/17) / zeta(4-8*4/17) )		#d=8/17

##############################################################################################

# ----> lemma-label{sum2log}

Prod_sum2log_l = 0.766633895034946  ### Precision: 3*10^9, Time: <360s, using C++ ---> int_double
Prod_sum2log_u = 0.766634119738678

deltaProd_sum2log_l = 5.41350358675289
deltaProd_sum2log_u = 5.42101612591179

Sum_sum2log_l = 0.749513593011507+gamma
Sum_sum2log_u = 0.74953459116131+gamma

B2_sum2log = RIF(1+(2^(1-delta)-1)/(2^(2-2*delta)+1))

j2_sum2log = RIF(1-B2_sum2log/(2+B2_sum2log))
k2_sum2log = RIF(log(2)*B2_sum2log/(B2_sum2log+2))
l2_sum2log = RIF(1+((2-B2_sum2log)*2^(1-delta)-B2_sum2log)/(2^(2-2*delta)+2^(1-delta)*(B2_sum2log-1)+B2_sum2log))

constant_om1 = RIF(Prod_sum2log_u / 2 + 1/log(Threshold2) * ( Sum_sum2log_u ) + A(1,delta) * (deltaProd_sum2log_u) / delta/ log(Threshold2)^2)
constant_om2 = RIF(j2_sum2log * Prod_sum2log_u / 2 + 1/log(Threshold2) * ( k2_sum2log + Sum_sum2log_u ) + l2_sum2log * A(1,delta) * (deltaProd_sum2log_u) / delta/ log(Threshold2)^2)

constant_om1BIP = RIF(Prod_sum2log_u / 2 + 1/log(Thresholdbip2) * ( Sum_sum2log_u ) + A(1,delta) * (deltaProd_sum2log_u) / delta/ log(Thresholdbip2)^2)
constant_om2BIP = RIF(j2_sum2log * Prod_sum2log_u / 2 + 1/log(Thresholdbip2) * ( k2_sum2log + Sum_sum2log_u ) + l2_sum2log * A(1,delta) * (deltaProd_sum2log_u) / delta/ log(Thresholdbip2)^2)

##############################################################################################

# ----> lemma-label{sumvarp}  

Prod_sumvarp_l = Trunc(1.94359643387259,dlong)  ### Precision: 3*10^9, Time: <360s, using C++ ---> int_double
Prod_sumvarp_u = Trunc(1.94359649909918,dlong)

deltaProd_sumvarp_l=Trunc(35.1592481967473,dlong)  ### delta2 = 11/8
deltaProd_sumvarp_u=Trunc(35.3602424406516,dlong)

deltaProdS1_sumHalf_l=Trunc(8.18840472503388,dlong)  ### delta2 = 1
deltaProdS1_sumHalf_u=Trunc(8.18840524101289,dlong)

deltaProdS2_sumHalf_l=Trunc(17.3092458701807,dlong)  ### delta2 = 5/4
deltaProdS2_sumHalf_u=Trunc(17.3092458701807,dlong)

deltaProdS3_sumHalf_l=Trunc(10.4996429321967,dlong) ### delta2 = 1.101 
deltaProdS3_sumHalf_u=Trunc(10.4996436034292,dlong)

deltaProdS4_sumHalf_l=Trunc(68.8610486654603,dlong)  ### delta2 = 23/16
deltaProdS4_sumHalf_u=Trunc(81.976134467179,dlong)

deltaProdS5_sumHalf_l=Trunc(123.205479993632,dlong)  ### delta2 = 47/32
deltaProdS5_sumHalf_u=Trunc(481.798441901477,dlong)

f2p = RIF(1-2/(2^2-2+1)) 
g2p = RIF(1+(2^2-2^delta2-4*2+2)/(2^(2-delta2)*(2-1)^2+2^delta2+2*2-1))

constant1_m1 = RIF(Prod_sumvarp_u + A(2,delta2)*deltaProd_sumvarp_u/(5*Thresholdv2/10)^((delta2-1)))
constant1_m2 = RIF(Prod_sumvarp_u * f2p + A(2,delta2)*deltaProd_sumvarp_u/(5*Thresholdv2/10)^((delta2-1)) * g2p)

constant2_m1 = 3.65284002906416
constant2_m2 = 1.23963002179812

constant_m1 = max(constant1_m1,constant2_m1)
constant_m2 = max(constant1_m2,constant2_m2)

constant_m1BIP = constant_m1
constant_m2BIP = constant_m2

##############################################################################################

# ----> lemma-label{Ss1}  

deltaProd_sumSs1_l=Trunc(26.3023533198884,dlong)
deltaProd_sumSs1_u=Trunc(26.320623342386,dlong)

f2p = RIF(1-2/(2^2-2+1)) 
h2p = RIF(1+(2^2-2^(1+delta)-4*2+2)/(2^(1-delta)*(2-1)^2+2^(1+delta)+2*2-1))

constant1_w1 = RIF(Prod_sumvarp_u + A(1,delta)*(2-delta)*deltaProd_sumSs1_u/(1-delta)/(Thresholdv2)^(delta))
constant1_w2 = RIF(Prod_sumvarp_u * f2p + A(1,delta)*(2-delta)*deltaProd_sumSs1_u/(1-delta)/(Thresholdv2)^(delta) * h2p)

constant2_w1 = 1.94359647941059
constant2_w2 = 0.647865567529124

constant_w1 = max(constant1_w1,constant2_w1)
constant_w2 = max(constant1_w2,constant2_w2)

constant_w1BIP = constant_w1
constant_w2BIP = constant_w2

#DELTA PRODUCTS

Error_sumSs1_1= RIF(A(1,1/30) * (2-1/30) / (1-1/30) * 8.83934258158701 / Threshold2^(1/30))			#d=1/30, Hd=[8.83934205871588,8.83934258158701]
Error_sumSs1_2 = RIF(A(1,1/9) * (2-1/9) / (1-1/9) * 10.7976387053896 / Threshold2^(1/9))				#d=1/9, Hd=[10.7976379411273,10.7976387053896]
Error_sumSs1_3 = RIF(A(1,1/7) * (2-1/7) / (1-1/7) * 11.8428635283205 / Threshold2^(1/7))			#d=1/7, Hd=[11.8428619436714,11.8428635283205]
Error_sumSs1_4 = RIF(A(1,1/5) * (2-1/5) / (1-1/5) * 14.2781215093248 / Threshold2^(1/5))			#d=1/5, Hd=[14.2781040503951,14.2781215093248]
Error_sumSs1_5 = RIF(A(1,1/3) * (2-1/3) / (1-1/3) *  26.3206233423862 / Threshold2^(1/3))			#d=1/3, Hd=[26.3023533198884,26.3206233423862]
Error_sumSs1_6 = RIF(A(1,2/5) * (2-2/5) / (1-2/5) * 44.8184727295661 / Threshold2^(2/5))			#d=2/5, Hd=[43.8778284108464,44.8184727295661]
Error_sumSs1_7 = RIF(A(1,4/9) * (2-4/9) / (1-4/9) * 100.006373211557 / Threshold2^(4/9))			#d=4/9, Hd=[76.6825997339701,100.006373211557]
Error_sumSs1_8 = RIF(A(1,8/17) * (2-8/17) / (1-8/17) * 618.740666005107 / Threshold2^(8/17))		#d=8/17, Hd=[128.735118453479,618.740666005107]

##############################################################################################

# ----> lemma-label{sum.1/2}  

delta_sumHalf=1/3

F2=RIF(1-(sqrt(2)+1)/(2^(3/2)-2+2))
G2=RIF(log(2)/(2-2*sqrt(2)+2))
D2=RIF(1+(2-4*sqrt(2)-2^delta_sumHalf+2)/((sqrt(2)-1)^2*2^(1-delta_sumHalf)+2*sqrt(2)+2^delta_sumHalf-1))

Prod_sumHalf_l = 15.0333977306198  ### Precision: 3*10^9, Time: <360s, using C++ ---> int_double
Prod_sumHalf_u = 15.0337644348976

deltaProd1_sumHalf_l = 10.7125980293757
deltaProd1_sumHalf_u = 10.7200446392398

deltaProd2_sumHalf_l= 150.759502870216
deltaProd2_sumHalf_u= 165.591131996911 

Sum_sumHalf_l = -1.73634511834745+gamma
Sum_sumHalf_u = -1.73581610256903+gamma

#DELTA PRODUCTS

Error_sumHalf_1= RIF(A(1,1/3) * 1792.86936752771 / Threshold^(1/3))			#d=1/3, Hd=[1612.51548737184,1792.86936752771],[10.7125980293757,10.7200446392398]-[150.759502870216,165.591131996911]
Error_sumHalf_2 = RIF(A(1,11/32) * 2205.43552891345 / Threshold^(11/32))		#d=11/32, Hd=[1913.11941809751,2205.43552891345],[10.7061052639316,10.7186098449943]-[178.694246967918,205.757608571173]
Error_sumHalf_3 = RIF(A(1,6/17) * 2702.54807903974 / Threshold^(6/17))			#d=6/17, Hd=[2246.17435290147,2702.54807903974],[10.6954578407999,10.7152816957671]-[210.011986988813,252.214375297975]
Error_sumHalf_4 = RIF(A(1,3/8) * 4914.35562349626 / Threshold^(3/8))			#d=3/8, Hd=[3448.13446436226,4914.35562349626],[10.6484790351858,10.7093722208825]-[323.814739454205,458.883632218294]
Error_sumHalf_5 = RIF(A(1,8/21) * 5991.07302993625 / Threshold^(8/21))			#d=8/21, Hd=[3918.3769356223,5991.07302993625],[10.6297861027956,10.7126223292616]-[368.622369042004,559.253639846107]
Error_sumHalf_6 = RIF(A(1,9/23) * 8918.17703847463 / Threshold^(9/23))			#d=9/32, Hd=[4965.43289805425,8918.17703847463],[10.5897960141305,10.7321784325023]-[468.888436701578,830.975471994204]
Error_sumHalf_7 = RIF(A(1,2/5) * 13336.5474051997 / Threshold^(2/5))			#d=2/5, Hd=[6155.39804245799,13336.5474051997],[10.5473705964455,10.7735284010179]-[583.595502421462,1237.899684187]
Error_sumHalf_8 = RIF(A(1,4/9) * 1169421.21876513 / Threshold^(4/9))			#d=4/9, Hd=[25244.9803916096,1169421.21876513],[10.0427264910237,13.0976618077657]-[2513.75763485879,89284.7315748967]

# ----> lemma-label{sum.1/2threshold}  

constant1_threshold_1half = 14.1135189389139 
constant2_threshold_1half = 2.08929191526059 

constant_x1v = RIF(Prod_sumHalf_u + Sum_sumHalf_u/log(Threshold3) + A(1,delta_sumHalf)*deltaProd1_sumHalf_u*deltaProd2_sumHalf_u/Threshold3^(delta_sumHalf)/log(Threshold3))

constant_x2 = RIF(Prod_sumHalf_u *F2+ (Sum_sumHalf_u+G2)/log(Threshold3) + D2*A(1,delta_sumHalf)*deltaProd1_sumHalf_u*deltaProd2_sumHalf_u/Threshold3^(delta_sumHalf)/log(Threshold3))

constant_x1 = RIF( constant_x2 * (1+1/(sqrt(2)-1)^2))

##############################################################################################

# ----> lemma-label{Ss2}  

delta_Ss2=1/3

F_Ss2 = RIF(1-(2-1)^2/(2^2*(sqrt(2)-1)^2+2*2^(3/2)-3*2+1))
H_Ss2 = RIF(1+(2^2-4*2^(3/2)-(2-1)*2^delta_Ss2+3*2)/(2^(2-delta_Ss2)*(sqrt(2)-1)^2+2*2^(3/2)+(2-1)*2^delta_Ss2-2*2))

Prod_Ss2_l = 5.69431329873424  ### Precision: 3*10^9, Time: <360s, using C++ ---> int_double
Prod_Ss2_u = 5.69441191918336

deltaProd1_Ss2_l = 1.70954296717661
deltaProd1_Ss2_u = 1.71048525255215

deltaProd2_Ss2_l = deltaProd2_sumHalf_l
deltaProd2_Ss2_u = deltaProd2_sumHalf_u 

#DELTA PRODUCTS

Error_sumHalf_1= RIF(A(1,1/3) * 1792.86936752771 / Threshold^(1/3))			#d=1/3, Hd=[1612.51548737184,1792.86936752771],[10.7125980293757,10.7200446392398]-[150.759502870216,165.591131996911]

Error_sumHalf_1BIP = RIF(A(1,1/3) * 1792.86936752771 / Thresholdbip^(1/3))	

##############################################################################################

# ----> lemma-label{sum2.1/2}  

delta_sum2Half=1/3

X22 = RIF(1-1/(2-2*sqrt(2)+2))
Y22 = RIF(1+(2-4*sqrt(2)-2^(delta_sum2Half)+2)/(2^(1-delta_sum2Half)*(sqrt(2)-1)^2+2*sqrt(2)+2^delta_sum2Half-1))

Prod_sum2Half_l = 15.0334119036497  ### Precision: 3*10^9, Time: <360s, using C++ ---> int_double
Prod_sum2Half_u = 15.0336721156811

deltaProd_sum2Half_l = 1615.05735275972
deltaProd_sum2Half_u = 1775.90360702186

constant1_threshold2_1half = 16.7682417501771
constant2_threshold2_1half = 2.50640699370728

constant2_x1v = RIF(Prod_sum2Half_u + deltaProd_sum2Half_u * A(1,delta_sum2Half) * (2-delta_sum2Half)/(1-delta_sum2Half) / Threshold3^delta_sum2Half)
constant2_x2 = RIF(X22 * Prod_sum2Half_u + Y22 * deltaProd_sum2Half_u * A(1,delta_sum2Half) * (2-delta_sum2Half)/(1-delta_sum2Half) / Threshold3^delta_sum2Half)
constant2_x1 = RIF( constant2_x2 * (1+1/(sqrt(2)-1)^2))

##############################################################################################

# ----> lemma-label{Ss3}  

delta_Ss3 = 1/3

F_Ss3 = RIF(1-(2^(3/2)+2-sqrt(2)-1)/(2^(5/2-theta)*(2^theta-1)+2+2^(3/2-theta)-sqrt(2)-1))
H_Ss3 = RIF(1+(2^(3/2-delta_Ss3)-2*2^(3/2-theta-delta_Ss3)-2^(1-delta_Ss3)-sqrt(2)-1)/(2^(3/2-2*delta_Ss3)*(2-2^(1-theta))+2^(3/2-theta-delta_Ss3)+2^(1-delta_Ss3)+sqrt(2)+1)) 

Prod_Ss3_l = 2.25359602570734
Prod_Ss3_u = 2.25361568426274

deltaProd_Ss3_l = 55.1999077928566
deltaProd_Ss3_u = 57.9152802733073

Error_Ss3 = A(1,delta_Ss3)*deltaProd_Ss3_u*(1-delta_Ss3)/(1/2-delta_Ss3) / Threshold^delta_Ss3
Error_Ss3BIP = A(1,delta_Ss3)*deltaProd_Ss3_u*(1-delta_Ss3)/(1/2-delta_Ss3) / Thresholdbip^delta_Ss3

##############################################################################################

# ----> lemma-label{Parity} 

Prod_Parity_l = 0.660161800282638  ### Precision: 3*10^9, Time: <360s, using C++ ---> int_double
Prod_Parity_u = 0.660161816820513

deltaProd_Parity_l = 1.07964726493134
deltaProd_Parity_u = 1.08039615954305

constant_qeven = RIF(4 * (1+1/2^(1-delta))*A(1,delta)*0.660161816820513*1.08039615954305)
constant_qodd = RIF(2^(1+delta) * (1+1/2^(1-delta))*A(1,delta)*0.660161816820513*1.08039615954305)

##############################################################################################

# ----> lemma-label{Ss1Log} _______ N  E  W  

W1 = 10^(12)

f2_2_Ss1Log = RIF( 1 - 1/(2^(2-2*theta)*(2^theta-1)^2+1) )
f2_3_Ss1Log = RIF( 1 - 1/(2^(3/2-2*theta)*(2^theta-1)^2+1) )

Prod2_Ss1Log_l = 2.93392267813336  ### Precision: 3*10^9, Time: <360s, using C++ ---> int_double
Prod2_Ss1Log_u = 2.93392292726709

Prod3_Ss1Log_l = 5.30171336966289 ### Precision: 3*10^9, Time: <360s, using C++ ---> int_double
Prod3_Ss1Log_u = 5.30175957628429

constant_y1 = RIF( 3 * Prod2_Ss1Log_u + ( 4/e/log(W1) + 16/e^2/log(W1)^2 ) * 5.30175957628429 )

constant_y2 = RIF( 3 * f2_2_Ss1Log * Prod2_Ss1Log_u + ( 4/e/log(W1) + 16/e^2/log(W1)^2 ) * f2_3_Ss1Log * 5.30175957628429 )

##############################################################################################

# ----> lemma-label{Ss2Log} _______ N  E  W

f2_0_Ss2Log=RIF(1-(2-1)/(2^(2-2*theta)*(2^theta-1)^2+2*2^(1-theta)-2^(1-2*theta)-1))
s2_Ss2Log=RIF(log(2)/(2^(1-2*theta)*(2^theta-1)^2+1))
f2_d_Ss2Log=RIF(1+(2^(2*theta)-4*2^theta-2^(2*theta+delta-1)+2)/(2^(1-delta)*(2^theta-1)^2+2*2^theta+2^(2*theta+delta-1)-1))

Prod1_Ss2Log_u = 2.0881396785312
Prod1_Ss2Log_l = 2.08813954913666  ### Precision: 3*10^9, Time: <360s, using C++ ---> int_double

Sum_Ss2Log_l = 0.282661139651601+gamma  ### Precision: 3*10^9, Time: <360s, using C++ ---> int_double
Sum_Ss2Log_u = 0.282661147670366+gamma

deltaProd_Ss2Log_l = 30.0384493826561
deltaProd_Ss2Log_u = 30.0384508591214

constant_z1 = RIF( 1*(2*A(1,delta)*deltaProd_Ss2Log_u/log(10^(12)) / (1-choice)^2 / 10^(12*delta*choice) + A(1,delta)*deltaProd_Ss2Log_u/log(10^(12)) / (1-(2/delta)*choice / (1-choice) / log(10^(12)) )^2) + 1*( choice / (1-choice) + (Sum_Ss2Log_u + 0) / log(10^(12)) ) * Prod1_Ss2Log_u )

constant_z2 = RIF( f2_d_Ss2Log * (2*A(1,delta)*deltaProd_Ss2Log_u/log(10^(12)) / (1-choice)^2 / 10^(12*delta*choice) + A(1,delta)*deltaProd_Ss2Log_u/log(10^(12)) / (1-(2/delta)*choice / (1-choice) / log(10^(12)) )^2) + f2_0_Ss2Log * ( choice / (1-choice) + (Sum_Ss2Log_u + s2_Ss2Log) / log(10^(12)) ) * Prod1_Ss2Log_u )

##############################################################################################

# ----> lemma-label{Ss3Log} 

k3=3/2 

deltaSs3Log = delta

f2_0_Ss3Log = RIF(1-(2-1)^2/((2^theta-1)^2*2^(3-2*theta)-2^(2-2*theta)+2*2^(2-theta)-2*2+1))
s2_Ss3Log = RIF(log(2)*(2-1)/(2^(2-2*theta)*(2^theta-1)^2+2-1))
f2_d_Ss3Log = RIF(1+(2^2-(2-1)*2^deltaSs3Log-4*2^(2-theta)+2+2*2^(2-2*theta))/(2^(3-2*theta-deltaSs3Log)*(2^theta-1)^2+2*2^(2-theta)+(2-1)*2^deltaSs3Log-2-2^(2-2*theta)))

Prod_Ss3Log_l = 1.05888883358328
Prod_Ss3Log_u = 1.05888895802532

Sum_Ss3Log_l = 0.698779325829545 + gamma ### Precision: 6*10^9, Time: <360s, using C++ ---> int_double
Sum_Ss3Log_u= 0.698779357125726 + gamma

deltaProd_Ss3Log_l = 9.13490825011212
deltaProd_Ss3Log_u = 9.14498162967894

constant_psi1 = RIF( Prod_Ss3Log_u *  ( 1/(k3-1)/sqrt(12*log(10)) + 1/(log(10^12)^(3/2)) * ( Sum_Ss3Log_u - 2 + 0 + A(1,deltaSs3Log) * deltaProd_Ss3Log_u * (2^(k3+1)+1) ) ) )

constant_psi2 = RIF( f2_0_Ss3Log*Prod_Ss3Log_u *  ( 1/(k3-1)/sqrt(12*log(10)) + 1/(log(10^12)^(3/2)) * ( Sum_Ss3Log_u - 2 + s2_Ss3Log + A(1,deltaSs3Log) * deltaProd_Ss3Log_u * (2^(k3+1)+1)*f2_d_Ss3Log ) ) )

##############################################################################################
### NEW STUFF ################################################################################
##############################################################################################

##############################################################################################

#--------------------

Ram_new_cst = 0.43
C2_v2_1= 3*9/70
C2_v2_2= 0.406722117422497

CONSTANT1 = RIF( max( 6*C_sum2/pi^2, Ram_new_cst ) )
CONSTANT_ALT = RIF( max(4*C_sum2_22/pi^2, max(C2_v2_1,C2_v2_2)) )

cst_crucial = CONSTANT1

CONSTANT2 = RIF( (1-1/sqrt(2)) * max( 4*C_sum2_22/pi^2, C2_v2_1,C2_v2_2 ) )

CONSTANT_RAM = RIF( (1-1/sqrt(2)) * (CONSTANT1+CONSTANT2/(sqrt(2)-1)) )

#--------------------

Prod_Ram_Lower = 9.37522491513744
Prod_Ram_Upper = 9.37530668903219

#--------------------

v0 = 2
v1 = 3
v2 = 5
v4 = 7
v1_2 = v0*v1
v3_2 = v0*v2
v7_2 = v0*v4

##############################################################################################

# ----> lemma-label{S1}

constant_X1 = RIF( constant_x1 * constant2_x1 * (1-choice) + (constant_x1 * constant2_x1 + constant2_x1) / log(Threshold) )
constant_X2 = RIF( constant_x2 * constant2_x2 * (1-choice) + (constant_x2 * constant2_x2 + constant2_x2) / log(Threshold) )

constant_c1=RIF(sqrt(constant_y1*constant_z1/constant_X1)*(Threshold)^((1-choice)/2)/log(Threshold)/389)
constant_c2=RIF(sqrt(constant_y2*constant_z2/constant_X2)*(Threshold)^((1-choice)/2)/log(Threshold)/389 * 2^theta*(sqrt(2)-1)/(2^theta-1)/sqrt(2))

constant1_errS1=RIF((1+constant_c1)*constant_X1*log(Threshold)^2/Threshold^(1-choice) + (1+1/constant_c1)*constant_y1*constant_z1/389^2)
constant2_errS1=RIF((1+constant_c2)*constant_X2*log(Threshold)^2*2/Threshold^(1-choice)/(sqrt(2)-1)^2+ (1+1/constant_c2)*constant_y2*constant_z2*2^(2*theta)/389^2/(2^theta-1)^2)

constant_X1BIP = RIF( constant_x1 * constant2_x1 * (1-choice) + (constant_x1 * constant2_x1 + constant2_x1) / log(Thresholdbip) )
constant_X2BIP = RIF( constant_x2 * constant2_x2 * (1-choice) + (constant_x2 * constant2_x2 + constant2_x2) / log(Thresholdbip) )

constant_c1BIP = RIF(sqrt(constant_y1*constant_z1/constant_X1)*(Thresholdbip)^((1-choice)/2)/log(Thresholdbip)/389)
constant_c2BIP = RIF(sqrt(constant_y2*constant_z2/constant_X2)*(Thresholdbip)^((1-choice)/2)/log(Thresholdbip)/389 * 2^theta*(sqrt(2)-1)/(2^theta-1)/sqrt(2))

constant1_errS1BIP = RIF((1+constant_c1BIP)*constant_X1*log(Thresholdbip)^2/sqrt(Thresholdbip) + (1+1/constant_c1BIP)*constant_y1*constant_z1/389^2)
constant2_errS1BIP = RIF((1+constant_c2BIP)*constant_X2*log(Thresholdbip)^2*2/sqrt(Thresholdbip)/(sqrt(2)-1)^2+ (1+1/constant_c2BIP)*constant_y2*constant_z2*2^(2*theta)/389^2/(2^theta-1)^2)

##############################################################################################

# ----> lemma-label{S2}

constant1_errS2 = RIF(constant_qodd + constant_m1 * constant_w1)

constant2_errS2 = RIF(constant_qeven + 2^2 * constant_m2 * constant_w2)

constant1_errS2BIP = RIF(constant_qodd + constant_m1BIP * constant_w1BIP)

constant2_errS2BIP = RIF(constant_qeven + 2^2 * constant_m2BIP * constant_w2BIP)

##############################################################################################

# ----> lemma-label{S3}

constant1_errS3 =  RIF(2 * Error_sum1 * constant_g10 / 4)
constant2_errS3 = RIF(2 * 2 * Error_sum1 * A2_sumvar1log * constant_g20 / 4) 

constant_o1 = RIF(2 * ( gamma * log(Threshold) / (Threshold)^(1/4) + 1 / 103))
constant_o2 = RIF(2 * ( sqrt(2)/(sqrt(2)-1) * gamma * log(Threshold) / (Threshold)^(1/4) + 2^theta / (2^theta-1) * 1 / 103))

constant1_errS3BIP =  RIF(2 * Error_sum1 * constant_g10BIP / 4)
constant2_errS3BIP = RIF(2 * 2 * Error_sum1 * A2_sumvar1log * constant_g20BIP / 4) 

constant_o1BIP = RIF(2 * ( gamma * log(Thresholdbip) / (Thresholdbip)^(1/4) + 1 / 103))
constant_o2BIP = RIF(2 * ( sqrt(2)/(sqrt(2)-1) * gamma * log(Thresholdbip) / (Thresholdbip)^(1/4) + 2^theta / (2^theta-1) * 1 / 103))

##############################################################################################

# ----> label{estimm}

constant_n1 = RIF(2 * Error_sum2 * constant_om1^2 / 2^4)
constant_n2 = RIF(2 * Error_sum2 * B2_sum2log * constant_om2^2 / 2^4)

constant_n1BIP = RIF(2 * Error_sum2 * constant_om1BIP^2 / 2^4)
constant_n2BIP = RIF(2 * Error_sum2 * B2_sum2log * constant_om2BIP^2 / 2^4)

##############################################################################################

constant_i1 = 3.83717
constant_j1 = 4.89606
constant_k1 = 0.000033536

constant_i2 = 4.99703
constant_j2 = 9.57182
constant_k2 = 0.0000615022

constant_e1 = RIF( 6 / pi^2 * (constant_i1 * log(Threshold)^2 / 2 / sqrt(Threshold) + (constant_i1 + constant_j1) * log(Threshold) / sqrt(Threshold) + 4 * constant_k1 ))
constant_e2 = RIF( 6 / pi^2 * 2 / 3 * (constant_i2 * log(Threshold)^2 / 2 / sqrt(Threshold) + (constant_i2 + constant_j2) * log(Threshold) / sqrt(Threshold) + 4 * constant_k2 ))

constant_e1BIP = RIF( 6 / pi^2 * 2 * constant_k1 )
constant_e2BIP = RIF( 6 / pi^2 * 2 / 3 * 2 * constant_k2 )

##############################################################################################

# ----> label{mainterm}

constant_integral1_l = -0.04951001463
constant_integral1_u = -0.04951000438

constant_integral2_l = 2.63481249177
constant_integral2_u = 2.63481271383

error_integral1 =  RIF ( constant_i1 * log(Threshold)/Threshold + ( constant_i1 + constant_j1 )/Threshold + 2 * constant_k1/log (10^(12)) )
error_integral2 =  RIF ( constant_i2 * log(Threshold)/Threshold + ( constant_i2 + constant_j2 )/Threshold + 2 * constant_k2/log (10^(12)) )

constant_s1_l = RIF(gamma - 6 / pi^2 * constant_integral1_u - error_integral1)
constant_s1_u = RIF(gamma - 6 / pi^2 * constant_integral1_l + error_integral1)

constant_s2_l = RIF(2 * (gamma + log(2) ) - 4 / pi^2 * constant_integral2_u - error_integral2)
constant_s2_u = RIF(2 * (gamma + log(2) ) - 4 / pi^2 * constant_integral2_l + error_integral2)

error_s1 = RIF( (constant_s1_u + constant_s1_l )/2 - 0.607)
error_s2 = RIF( (constant_s2_u + constant_s2_l )/2 - 1.472)

barrier_threshold_1 = 0.610576755017781
barrier_threshold_2 = 1.47308496508565

##############################################################################################

# ----> label{error01}

exbig = 8
barrier = 3
barrier2 = 8
bip= floor(e^exbig)

constant1_v1_Omega10_b3 = RIF( ( Prod_Ss2_u + A(1,delta_Ss2)*(2-delta_Ss2)*deltaProd1_Ss2_u *deltaProd2_Ss2_u/(1-delta_Ss2)/Threshold^delta_Ss2 ) * sqrt(log(barrier))/sqrt(barrier)/10^(exponent/2) )
constant2_v1_Omega10_b3 = RIF( (2*Prod_Ss3_u + A(1,delta_Ss3)*deltaProd_Ss3_u*(1-delta_Ss3)/(1/2-delta_Ss3) / Threshold^delta_Ss3)/ 389 / sqrt(log(barrier)) ) #U_2=U_0, u=3
constant1_v1_Omega10_b8 = RIF( ( Prod_Ss2_u + A(1,delta_Ss2)*(2-delta_Ss2)*deltaProd1_Ss2_u *deltaProd2_Ss2_u/(1-delta_Ss2)/Threshold^delta_Ss2 ) * sqrt(log(barrier2))/sqrt(barrier2)/10^(exponent/2) )
constant2_v1_Omega10_b8 = RIF( (2*Prod_Ss3_u + A(1,delta_Ss3)*deltaProd_Ss3_u*(1-delta_Ss3)/(1/2-delta_Ss3) / Threshold^delta_Ss3)/ 389 / sqrt(log(barrier2)) ) #U_2=U_0
constant1_v1_Omega10_bip = RIF( ( Prod_Ss2_u + A(1,delta_Ss2)*(2-delta_Ss2)*deltaProd1_Ss2_u *deltaProd2_Ss2_u/(1-delta_Ss2)/Thresholdbip^delta_Ss2 ) * sqrt(log(bip))/sqrt(bip)/10^(exponent/2) )
constant2_v1_Omega10_bip = RIF( (2*Prod_Ss3_u + A(1,delta_Ss3)*deltaProd_Ss3_u*(1-delta_Ss3)/(1/2-delta_Ss3) / Thresholdbip^delta_Ss3)/ 389 / sqrt(log(bip)) ) #U_2=U_0
constant3_v1_Omega10 = RIF( (2*Prod_Ss3_u + A(1,delta_Ss3)*deltaProd_Ss3_u*(1-delta_Ss3)/(1/2-delta_Ss3) / Threshold^delta_Ss3 )/ 389/10^6/sqrt(log(10^(12))) + constant_psi1/389^2 )
constant3_v1_Omega10bip = RIF( (2*Prod_Ss3_u + A(1,delta_Ss3)*deltaProd_Ss3_u*(1-delta_Ss3)/(1/2-delta_Ss3) / Thresholdbip^delta_Ss3 )/ 389/10^6/sqrt(log(10^(12))) + constant_psi1/389^2 )

constant1_v2_Omega10_b3 = RIF( 2 / (sqrt(2)-1)^2 * ( F_Ss2 * Prod_Ss2_u+ H_Ss2 * A(1,delta_Ss2)*(2-delta_Ss2)*deltaProd1_Ss2_u *deltaProd2_Ss2_u/(1-delta_Ss2)/Threshold^delta_Ss2 ) * sqrt(log(barrier)) / sqrt(barrier) / 10^(exponent/2) )
constant2_v2_Omega10_b3 = RIF( 2^(1/2+theta) / (sqrt(2)-1)/(2^theta-1) * (F_Ss3 * 2*Prod_Ss3_u + H_Ss3 * A(1,delta_Ss3)*deltaProd_Ss3_u*(1-delta_Ss3)/(1/2-delta_Ss3) / Threshold^delta_Ss3)/ 389 / sqrt(log(barrier)) ) #U_2=U_0, u=8
constant1_v2_Omega10_b8 = RIF( 2 / (sqrt(2)-1)^2 * ( F_Ss2 * Prod_Ss2_u+ H_Ss2 * A(1,delta_Ss2)*(2-delta_Ss2)*deltaProd1_Ss2_u *deltaProd2_Ss2_u/(1-delta_Ss2)/Threshold^delta_Ss2 ) * sqrt(log(barrier2)) / sqrt(barrier2) / 10^(exponent/2) )
constant2_v2_Omega10_b8 = RIF( 2^(1/2+theta) / (sqrt(2)-1)/(2^theta-1) * (F_Ss3 * 2*Prod_Ss3_u + H_Ss3 * A(1,delta_Ss3)*deltaProd_Ss3_u*(1-delta_Ss3)/(1/2-delta_Ss3) / Threshold^delta_Ss3)/ 389 / sqrt(log(barrier2)) ) #U_2=U_0, u=8
constant1_v2_Omega10_bip = RIF( 2 / (sqrt(2)-1)^2 * ( F_Ss2 * Prod_Ss2_u+ H_Ss2 * A(1,delta_Ss2)*(2-delta_Ss2)*deltaProd1_Ss2_u *deltaProd2_Ss2_u/(1-delta_Ss2)/Thresholdbip^delta_Ss2 ) * sqrt(log(bip)) / sqrt(bip) / 10^(exponent/2) )
constant2_v2_Omega10_bip = RIF( 2^(1/2+theta) / (sqrt(2)-1)/(2^theta-1) * (F_Ss3 * 2*Prod_Ss3_u + H_Ss3 * A(1,delta_Ss3)*deltaProd_Ss3_u*(1-delta_Ss3)/(1/2-delta_Ss3) / Thresholdbip^delta_Ss3)/ 389 / sqrt(log(bip)) ) #U_2=U_0, u=bip
constant3_v2_Omega10 = RIF( 2^(1/2+theta) / (sqrt(2)-1)/(2^theta-1) * (F_Ss3 * 2*Prod_Ss3_u + H_Ss3 * A(1,delta_Ss3)*deltaProd_Ss3_u*(1-delta_Ss3)/(1/2-delta_Ss3) / Threshold^delta_Ss3)/ 389/10^6/sqrt(log(10^(12))) + 2^(2*theta) / (2^theta-1)^2 * constant_psi2/389^2 )
constant3_v2_Omega10bip = RIF( 2^(1/2+theta) / (sqrt(2)-1)/(2^theta-1) * (F_Ss3 * 2*Prod_Ss3_u + H_Ss3 * A(1,delta_Ss3)*deltaProd_Ss3_u*(1-delta_Ss3)/(1/2-delta_Ss3) / Thresholdbip^delta_Ss3)/ 389/10^6/sqrt(log(10^(12))) + 2^(2*theta) / (2^theta-1)^2 * constant_psi2/389^2 )

constant1_v1_Omega11_b3 = constant1_v1_Omega10_b3
constant2_v1_Omega11_b3 = RIF( (2*Prod_Ss3_u + A(1,delta_Ss3)*deltaProd_Ss3_u*(1-delta_Ss3)/(1/2-delta_Ss3) / Threshold^delta_Ss3)/ 389 / sqrt(log(barrier)*barrier) ) #U_2=U_1
constant1_v1_Omega11_b8 = constant1_v1_Omega10_b8
constant2_v1_Omega11_b8 = RIF( (2*Prod_Ss3_u + A(1,delta_Ss3)*deltaProd_Ss3_u*(1-delta_Ss3)/(1/2-delta_Ss3) / Threshold^delta_Ss3)/ 389 / sqrt(log(barrier2)*barrier2) ) #U_2=U_1
constant1_v1_Omega11_bip = constant1_v1_Omega10_bip
constant2_v1_Omega11_bip = RIF( (2*Prod_Ss3_u + A(1,delta_Ss3)*deltaProd_Ss3_u*(1-delta_Ss3)/(1/2-delta_Ss3) / Thresholdbip^delta_Ss3)/ 389 / sqrt(log(bip)*bip) ) #U_2=U_1
constant3_v1_Omega11 = constant3_v1_Omega10
constant3_v1_Omega11bip = constant3_v1_Omega10bip

constant1_v2_Omega11_b3 = constant1_v2_Omega10_b3
constant2_v2_Omega11_b3 = RIF( 2^(1/2+theta) / (sqrt(2)-1)/(2^theta-1) * (F_Ss3 * 2*Prod_Ss3_u + H_Ss3 * A(1,delta_Ss3)*deltaProd_Ss3_u*(1-delta_Ss3)/(1/2-delta_Ss3) / Threshold^delta_Ss3)/ 389 / sqrt(log(barrier)*barrier) ) #U_2=U_1, u=3
constant1_v2_Omega11_b8 = constant1_v2_Omega10_b8
constant2_v2_Omega11_b8 = RIF( 2^(1/2+theta) / (sqrt(2)-1)/(2^theta-1) * (F_Ss3 * 2*Prod_Ss3_u + H_Ss3 * A(1,delta_Ss3)*deltaProd_Ss3_u*(1-delta_Ss3)/(1/2-delta_Ss3) / Threshold^delta_Ss3)/ 389 / sqrt(log(barrier2)*barrier2) ) #U_2=U_1, u=8
constant1_v2_Omega11_bip = constant1_v2_Omega10_bip
constant2_v2_Omega11_bip = RIF( 2^(1/2+theta) / (sqrt(2)-1)/(2^theta-1) * (F_Ss3 * 2*Prod_Ss3_u + H_Ss3 * A(1,delta_Ss3)*deltaProd_Ss3_u*(1-delta_Ss3)/(1/2-delta_Ss3) / Threshold^delta_Ss3)/ 389 / sqrt(log(bip)*bip) ) #U_2=U_1, u=bip
constant3_v2_Omega11 = constant3_v2_Omega10
constant3_v2_Omega11bip = constant3_v2_Omega10bip

constant_v1_Omega10_b3 = constant1_v1_Omega10_b3 + constant2_v1_Omega10_b3 + constant3_v1_Omega10 #U_2=U_0, u=3
constant_v2_Omega10_b3 = constant1_v2_Omega10_b3 + constant2_v2_Omega10_b3 + constant3_v2_Omega10

constant_v1_Omega11_b3 = constant1_v1_Omega11_b3 + constant2_v1_Omega11_b3 + constant3_v1_Omega11 #U_2=U_1, u=3
constant_v2_Omega11_b3 = constant1_v2_Omega11_b3 + constant2_v2_Omega11_b3 + constant3_v2_Omega11

constant_v1_Omega10_b8 = constant1_v1_Omega10_b8 + constant2_v1_Omega10_b8 + constant3_v1_Omega10 #U_2=U_0, u=8
constant_v2_Omega10_b8 = constant1_v2_Omega10_b8 + constant2_v2_Omega10_b8 + constant3_v2_Omega10

constant_v1_Omega11_b8 = constant1_v1_Omega11_b8 + constant2_v1_Omega11_b8 + constant3_v1_Omega11 #U_2=U_1, u=8
constant_v2_Omega11_b8 = constant1_v2_Omega11_b8 + constant2_v2_Omega11_b8 + constant3_v2_Omega11

constant_v1_Omega10_bip = constant1_v1_Omega10_bip + constant2_v1_Omega10_bip + constant3_v1_Omega10bip #U_2=U_0, u=8
constant_v2_Omega10_bip = constant1_v2_Omega10_bip + constant2_v2_Omega10_bip + constant3_v2_Omega10bip

constant_v1_Omega11_bip = constant1_v1_Omega11_bip + constant2_v1_Omega11_bip + constant3_v1_Omega11bip #U_2=U_1, u=8
constant_v2_Omega11_bip = constant1_v2_Omega11_bip + constant2_v2_Omega11_bip + constant3_v2_Omega11bip

##############################################################################################

# ----> label{boundBp}

constantp_Ap1_b3 = 1.59600978550163 
constantp_Ap2_b3 = 0.729869895920385

constant_Ap1_b3 = RIF((1+1/log(barrier))*constantp_Ap1_b3)
constant_Ap2_b3 = RIF((1+1/log(barrier))*constantp_Ap2_b3)

constantp_Ap1_b8 = 1.2464022956793
constantp_Ap2_b8 = 0.682450327755849

constant_Ap1_b8 = RIF((1+1/log(barrier2))*constantp_Ap1_b8)
constant_Ap2_b8 = RIF((1+1/log(barrier2))*constantp_Ap2_b8)

constantp_Ap1_bip = 1.21222063841589
constantp_Ap2_bip = 0.682450327755849

constant_Ap1_bip = constant_om1
constant_Ap2_bip = constant_om2

##############################################################################################

# ----> lemma-label{SigmaSum}

omega_1_v1_b3 = 2 * Error_sum2 * constant_Ap1_b3
omega_1_v2_b3 = B2_sum2log * 2 * Error_sum2 * constant_Ap2_b3

omega_2_v1 = RIF(2 * gamma * zeta(3/2)*1.36843279889225 * 1 )
omega_2_v2 = RIF(2 * gamma * zeta(3/2)*1.36843279889225 * k2_1half )

omega_3_v1_b3 = RIF(zeta(1+theta)*1.00724553394618 * 1/ 103 / sqrt(log(barrier)))
omega_3_v2_b3 = RIF(zeta(1+theta)*1.00724553394618 * k2_theta / 103 / sqrt(log(barrier)))

omega_1_v1_b8 = 2 * Error_sum2 * constant_Ap1_b8
omega_1_v2_b8 = B2_sum2log * 2 * Error_sum2 * constant_Ap2_b8

omega_3_v1_b8 = RIF( zeta(1+theta)*1.00724553394618 * 1/ 103 / sqrt(log(barrier2)))
omega_3_v2_b8 = RIF( zeta(1+theta)*1.00724553394618 * k2_theta / 103 / sqrt(log(barrier2)))

omega_1_v1_bip = 2 * Error_sum2 * constant_Ap1_bip
omega_1_v2_bip = B2_sum2log * 2 * Error_sum2 * constant_Ap2_bip

omega_3_v1_bip = RIF( zeta(1+theta)*1.00724553394618 * 1/ 103 / sqrt(log(bip)))
omega_3_v2_bip = RIF( zeta(1+theta)*1.00724553394618 * k2_theta / 103 / sqrt(log(bip)))

##############################################################################################

# ----> lemma-label{SigmaU1U0}

expc = 2
condition = 10^expc\

Lambda_v1_b3 = RIF( I_Error_sum1_u * 1 / condition / log(barrier)^2  + omega_1_v1_b3 / condition + omega_2_v1*sqrt(log(barrier)) / sqrt(barrier) + omega_3_v1_b3+ constant_v1_Omega10_b3) #u=3
Lambda_v2_b3 = RIF( I_Error_sum1_u * A2_sumvar1log / condition / log(barrier)^2  + omega_1_v2_b3 / condition +omega_2_v2*sqrt(log(barrier)) / sqrt(barrier) + omega_3_v2_b3+ constant_v2_Omega10_b3)

Lambda_v1_b8 = RIF( I_Error_sum1_u * 1 / condition / log(barrier2)^2  + omega_1_v1_b8 / condition + omega_2_v1*sqrt(log(barrier2)) / sqrt(barrier2) + omega_3_v1_b8+ constant_v1_Omega10_b8) #u=8
Lambda_v2_b8 = RIF( I_Error_sum1_u * A2_sumvar1log / condition / log(barrier2)^2  + omega_1_v2_b8 / condition + omega_2_v2*sqrt(log(barrier2)) / sqrt(barrier2) + omega_3_v2_b8+ constant_v2_Omega10_b8)

Lambda_v1_bip = RIF( I_Error_sum1_u * 1 / condition / log(bip)^2  + omega_1_v1_bip / condition +omega_2_v1*sqrt(log(bip)) / sqrt(bip)+ omega_3_v1_bip+ constant_v1_Omega10_bip) #u=8
Lambda_v2_bip = RIF( I_Error_sum1_u * A2_sumvar1log / condition / log(bip)^2  + omega_1_v2_bip / condition + omega_2_v2*sqrt(log(bip)) / sqrt(bip) + omega_3_v2_bip+ constant_v2_Omega10_bip)

##############################################################################################

# ----> lemma-label{moderate}

Pi_v1_b3 = RIF( (constant_v1_Omega11_b3 *  ( log(10^6) + barrier_threshold_1))^(1/2) )
Pi_v2_b3 = RIF( (constant_v2_Omega11_b3 * ( 2 * log(10^6) + barrier_threshold_2))^(1/2) )

Pi_v1_b8 = RIF( (constant_v1_Omega11_b8 *  ( log(10^6) - constant_s1_l + barrier_threshold_1))^(1/2) )
Pi_v2_b8 = RIF( (constant_v2_Omega11_b8 * ( 2 * log(10^6) - constant_s2_l + barrier_threshold_2))^(1/2) )

Pi_v1_bip = RIF( (constant_v1_Omega11_bip *  ( log(10^6) - constant_s1_l + barrier_threshold_1))^(1/2) )
Pi_v2_bip = RIF( (constant_v2_Omega11_bip * ( 2 * log(10^6) - constant_s2_l + barrier_threshold_2))^(1/2) )

##############################################################################################

# ----> lemma-label{BVsieve}

huge = expbip
big = exbig
denom = sqrt(12)

SieveErr0_v1 = RIF((constant1_errS1BIP + constant_e1BIP)/log(10^huge) + (constant1_errS2BIP / log(Thresholdbip)^4 + 2 * constant1_errS3BIP / log(Thresholdbip)^2 + constant_n1BIP)*log(10^huge)^4/10^(huge/6))
SieveErr0_v2 = RIF((constant2_errS1BIP + constant_e2BIP)/log(10^huge) + (constant2_errS2BIP / log(Thresholdbip)^4 + 2 * constant2_errS3BIP / log(Thresholdbip)^2 + constant_n2BIP)*log(10^huge)^4/10^(huge/6))

SieveErr1_v1 = RIF((constant1_errS1BIP + constant_e1BIP)/log(e^big*10^huge) + (constant1_errS2BIP / log(Thresholdbip)^4 + 2 * constant1_errS3BIP / log(Thresholdbip)^2 + constant_n1BIP)*log(e^big*10^huge)^4/e^(big/6)/10^(huge/6))
SieveErr1_v2 = RIF((constant2_errS1BIP + constant_e2BIP)/log(e^big*10^huge) + (constant2_errS2BIP / log(Thresholdbip)^4 + 2 * constant2_errS3BIP / log(Thresholdbip)^2 + constant_n2BIP)*log(e^big*10^huge)^4/e^(big/6)/10^(huge/6))

SieveLambdaErr_v1 = RIF(Lambda_v1_bip/sqrt(big))
SieveLambdaErr_v2 = RIF(Lambda_v2_bip/sqrt(big))

ConstantH1 = RIF((6/pi^2 + 2*sqrt(3)*(1-6/pi^2))^2)
ConstantH2 = RIF(4*(4/pi^2* + 2*(1-4/pi^2))^2)

DenomH1 = RIF(ConstantH1/denom^2)
DenomH2 = RIF(ConstantH2/(denom/2^(3/2))^2)

SieveTot_v1 = SieveErr0_v1 + SieveErr1_v1 + 2*SieveLambdaErr_v1 + DenomH1
SieveTot_v2 = SieveErr0_v2 + SieveErr0_v2 + 2*SieveLambdaErr_v2 + DenomH2 

\end{mysage}

%%%%% To ease editing, for IMPAN journals add:

\baselineskip=17pt

%%%%%%%%%%%%%%%%

\title{Explicit averages of square-free supported functions: to the edge of the convolution method}

\author{Sebastian Zuniga Alterman\\
Institut de Math\'{e}matiques de Jussieu\\ 
Universit\'{e} Paris Diderot P7\\ 
B\^atiment Sophie Germain, 8 Place Aur\'elie Nemours \\ 
75013 Paris, France\\
E-mail: sebastian.zuniga-alterman@imj-prg.fr}

\date{}

\maketitle

%% Classification and key words; note that the 2010 classification is used:

\renewcommand{\thefootnote}{}

\footnote{2020 \emph{Mathematics Subject Classification}: Primary 11N37; Secondary 11A25.}

\footnote{\emph{Key words and phrases}: explicit averages, arithmetic functions, convolution method}

\renewcommand{\thefootnote}{\arabic{footnote}}
\setcounter{footnote}{0}

%%%%%%%%

\begin{abstract}
We give a general statement of the convolution method so that one can provide explicit asymptotic estimations for all averages of square-free supported arithmetic functions that have a sufficiently regular order on the prime numbers and observe how the nature of this method gives error term estimations of order $X^{-\delta}$, where $\delta$ belongs to an open real positive set $I$. 
In order to have a better error estimation, a natural question is whether or not we can achieve an error term of critical order $X^{-\delta_0}$, where $\delta_0$, the critical exponent, is the right hand endpoint of $I$. We reply positively to that question by presenting a new method that improves qualitatively almost all instances of the convolution method under some regularity conditions; now, the asymptotic estimation of averages of well-behaved square-free supported arithmetic functions can be given with its critical exponent and a reasonable explicit error constant. 
We illustrate this new method by analyzing a particular average related to the work of Ramar\'e--Akhilesh (2017), which leads to notable improvements when imposing non-trivial coprimality conditions.
\end{abstract}

\section{Details and basic definitions}

In the present work, we use the {\em $O^*$ notation}: we write $f(X)=O^*(h(X))$, as $X\to a$ to indicate that $|f(X)|\leq h(X)$ in a neighborhood of $a$, where, in absence of precision, $a$ corresponds to $\infty$.
We also consider the {\em Euler $\varphi_s$ and Kappa $\kappa_s$} functions: let $s$ be any complex number, we define $\varphi_s:\mathbb{Z}_{>0}\to\mathbb{C}$ as $q\mapsto q^s\prod_{p|q}\left(1-\frac{1}{p^s}\right)$ and $\kappa_s:\mathbb{Z}_{>0}\to\mathbb{C}$ as $q\mapsto q^s\prod_{p|q}\left(1+\frac{1}{p^s}\right)$. 

{\em Computational details.} Every constant in this article has been estimated using interval arithmetic. 
Early numerical analysis was carried out using the ARB implementation, under the SageMath commands RBF and RIF, implemented in Python. We decided, however, to use Platt's implementation in C\texttt{++}, used for example in \cite{Pla16}, as it provides results with double precision, when compared to ARB, and at higher performance and faster speed. 

Throughout our calculations, we have set a precision order equal to $6\cdot 10^9$ and run a .cpp  script compiled with C\texttt{++}. We have also written a .ipynd  script (compiled by SageMath) to verify some of our results.

\section{Introduction}\label{Int}

The convolution method terminology was made popular by Ramar\'e in 1995, particularly in \cite[Lemma 3.2]{RA95}, where it was given in a somewhat hidden version with respect to the one we present in this article. It is a technique, already present in \cite{Mot78} and \cite{W1927}, among many other places, that relies upon a convolution identity and helps obtaining explicit estimations of averages of arithmetic functions, under some conditions. It is particularly meaningful when these arithmetic functions are supported on the square-free numbers, having a sufficiently regular behavior on all large prime numbers.  

While the convolution method provides the main term of a asymptotic expansion for the average of an arithmetic function with ease, it is at the remainder term where it shows its true potential, as it succeeds in giving a good enough estimation, explicit, for the error term: if the average is performed for the range $(0,X]$, where $X>0$, then the convolution methods gives error term explicit estimations of magnitude $X^{-\delta}$ when $\delta$ belongs to a maximal real open and positive interval $I$. 

Nevertheless, the nature of the convolution method does not allow one to obtain an error term estimation of magnitude $X^{-\delta_0}$ where $\delta_0$ is the right endpoint of $I$. Since it is usually a subject of interest in the explicit theory of numbers to improve error term magnitudes of expressions of interest, it is thus natural to ask whether or not one can provide, necessarily by a different method, an error term of critical order $\delta_0$ so that the overall estimation is qualitatively improved, going thus to the edge of the method of convolution.

We first present in \S\ref{dirichlet}, a special form of the convolution method involving sufficiently regular square-free supported functions, as shown in Theorem \ref{general}. As it relies upon some complex analytic facts, this method is related to a typical complex analytic approach for estimating the asymptotic expansion for the average of an arithmetic function by means of residue theory. 

Our main result, presented in \S\ref{improvement}, differs from complex analysis. In \S\ref{achieving}, we see how the use of some very particular estimations given in \S\ref{particular}, constitute the main ingredient to obtain reasonable explicit estimations of critical exponent in almost all cases where the convolution method may be applied with some conditions. Indeed, since our technique also relies upon the convergence of infinite products, some extra conditions on the regularity of the arithmetic function that is being averaged are needed, as Theorem \ref{general++} tells, and therefore there is a small range of functions that are not considered in our improvements, namely when the values of $\alpha$ and $\beta$ defined in Theorem \ref{general++} have a difference of absolute value smaller or equal than $\frac{1}{2}$. However, as most of the applications we mention throughout this article do not involve that missing case, we then claim that every one of these ones are improved up to their critical exponent.

Previous work towards the obtention of error terms of critical exponent can be found, on some particular averages, in \cite{Bu14} and \cite{W1927}. In \cite{RA13} and \cite{RA19}, the obtention of the critical exponent is carried out by a completely different approach, using some results known as \emph{the covering remainder lemma} and \emph{the unbalanced Dirichlet hyperbola formula} as well as strong explicit bounds on some summatory functions involving the M\"obius functions that, unlike our case of study, do oscillate. Furthermore, it is important to point out that whereas a similar path as in \cite{RA13} or \cite{RA19} could have been followed, these results consider specific properties of the functions that are being averaged and they are thus not easy to generalize to a broader class of functions. This is the reason why \cite[Thm. 1.2]{RA13} improves on the classic convolution method result presented in Corollary \ref{corollary} $\mathbf{(a)}$ but still requires the convolution method to estimate related averages of less simple arithmetic functions; for example, with the result we present in Theorem \ref{general++}, one can now immediately derive stronger estimations for \cite[Lemmas 7.1, 7.2, 7.6, 7.7, 7.8, 7.9]{RA13} that may lead to further improvements on the cited article of Ramar\'e--Akhilesh. In that aspect, our result might help as a reference for further improvements on many places where the convolution method is employed; it read as follows. 

\begin{theorem*} 
  Let $X>0$, be a real number and $q$ a positive integer. Consider a multiplicative function $f:\mathbb{Z}^+\to\mathbb{C}$ such that for every prime number $p$ satisfying $(p,q)=1$, we have $f(p)=\frac{1}{p^{\alpha}}+O\left(\frac{1}{p^{\beta}}\right)$, where $\alpha$, $\beta$ are real numbers satisfying $\beta>\alpha$, $\beta-\alpha>\frac{1}{2}$. Then there exists a constant $\mathrm{W}_{\alpha}^q>0$ such that
 \begin{align*}
 \sum_{\substack{\ell\leq X\\ (\ell,q)=1}}\mu^2(\ell)f({\ell})=F_\alpha^{q}(X)+\begin{cases}
O^*\left(\mathrm{W}_\alpha^q\ X^{\frac{1}{2}-\alpha}\right),\quad&\text{ if }\alpha\neq\frac{1}{2},\\
O^*\left(\mathrm{W}_\alpha^q\ \log(X)\right),\quad&\text{ if }\alpha=\frac{1}{2},\\
\end{cases}
 \end{align*} 
 where  
 \begin{align*} 
 F_\alpha^q(X)&=\frac{M_{\alpha}^{q}\zeta(\alpha)\varphi_\alpha(q)}{q^\alpha}-\frac{N_{\alpha}^q\varphi(q)}{(\alpha-1)q}\frac{1}{X^{\alpha-1}},\quad&&\text{if\ }\alpha>\frac{1}{2},\ \alpha\neq 1,\\
 F_1^q(X)&=\frac{M_{1}^{q}\varphi(q)}{q}\left(\log\left(X\right)+T_f^q+\gamma+\sum_{p|q}\frac{\log(p)}{p-1}\right),\\
&\phantom{xxxxxxxxx}T_{f}^{q}=\sum_{p\nmid q}\frac{\log(p)(1-(p-2)f(p))}{(f(p)+1)(p-1)},\\
F_\alpha^q(X)&=\frac{M_{\alpha}^{q}\varphi(q)}{(1-\alpha)q}X^{1-\alpha},\quad&&\text{if\ }\alpha\leq\frac{1}{2},
 \end{align*}
and, 
\begin{align*}
 M_{\alpha}^{q}&=\begin{cases}
\prod_{p\nmid q}\left(1-\frac{1-f(p)p^\alpha+f(p)}{p^{\alpha}}  \right),\quad&\text{ if }\alpha>\frac{1}{2},\\
N_{\alpha}^{q},\quad&\text{ if }\alpha\leq\frac{1}{2},
\end{cases}\\
N_{\alpha}^{q}&=\prod_{p\nmid q}\left(1-\frac{p^{1-\alpha}-f(p)p+f(p)}{p^{2-\alpha}} \right).
\end{align*}
 \end{theorem*}

As an application of the above theorem, we deduce how the improvement on the convolution method produces better savings on the error term constant of $\sum_{\substack{\ell\leq X\\(\ell,q)=1}}\frac{\mu^2(\ell)}{\varphi(\ell)}, X>0, q\in\mathbb{Z}_{>0}$ than the one in \cite[Thm. 1.1]{RA13} , when prime coprimality conditions are introduced. This situation is examined in \S\ref{Cop}, and we have for instance the improvement on the constant \sage{Trunc(5.9*21/25,3)}, given in \cite[Thm. 1.1]{RA13}, by $\sage{Upper(CONSTANT_RAM*Prod_Ram_Upper*CONSTANT2/CONSTANT_RAM,digits)}$, according to the following result.

\begin{lemma*}
Let $X>0$, then
\begin{equation*}
\sum_{\substack{\ell\leq X\\(\ell,2)=1}}\frac{\mu^2(\ell)}{\varphi(\ell)}=\frac{1}{2}\left(\log\left(X\right)+\mathfrak{a}_2\right)+O^*\left(\frac{\sage{Upper(CONSTANT_RAM*Prod_Ram_Upper*CONSTANT2/CONSTANT_RAM,digits)}}{\sqrt{X}}\right),
\end{equation*}
where $\mathfrak{a}_2=\sage{Trunc(C_sum1+log(2)/2,digits)}\ldots $. 
\end{lemma*}

\section{A special version of the method of convolution}\label{dirichlet}

In the convolution method, it is crucial to preserve regularity conditions, that is, conditions that do not impose specific ranges other than the variable itself being a positive integer, under, perhaps, some coprimality restrictions. 

To put an example, when one carries out a summation on a variable $e\in\mathbb{Z}_{>0}$ such that $e\leq \frac{X}{d}$ for certain real number $X>0$ and a positive integer $d$, it is often implicitly assumed that $\frac{X}{d}\geq 1$, so that the set $\{e\in\mathbb{Z}_{>0}, \ e\leq\frac{X}{d}\}$ is not empty. If $d$ is itself a variable, that means that we have the range condition $\{d\leq X\}$ on the variable $d$. Hence, if we are able to estimate asymptotically a summation on the variable $e\in\mathbb{Z}_{>0}$ such that $e\leq \frac{X}{d}$, regardless of whether or not an empty condition sum is performed, that is an \emph{empty sum}, then the range condition on the variable $d$ will be absent.

\subsection{Regularity conditions: estimating empty summations}

\begin{lemma}\label{recip} 
Let $\alpha\in \mathbb{R}^+\setminus \{1\}$ and $X>0$. Then
\begin{equation*}
\sum_{n\leq X}\frac{1}{n^\alpha}=\zeta(\alpha)-\frac{1}{(\alpha-1)X^{\alpha-1}}+O^*\left(\frac{1}{X^{\alpha}}\right).
\end{equation*}
\end{lemma} 

\begin{proof}
  By definition of $\zeta(s)$ for $\Re(s)>1$, and by analytic continuation
  for all $s\ne 1$ with $\Re(s)>0$,
  \begin{equation}\label{EulerMac}
 \zeta(s) - \frac{1}{(s-1) X^{s-1}} - \sum_{n\leq X} \frac{1}{n^s} =
  \sum_{n=1}^\infty 
  \left(\int_{n-1}^{n} \frac{dx}{(X+x)^s} - \frac{1}{(\lfloor X\rfloor+n)^s}\right).\end{equation}
  Set $s=\alpha$; clearly $(\lfloor X\rfloor+n)^{-\alpha} \geq (X+n)^{-\alpha}$ and by convexity of
  $t\mapsto \frac{1}{t^\alpha}$,
  $\int_{n-1}^{n} \frac{dx}{(X+x)^\alpha} \leq \frac{1}{2} \left(
  \frac{1}{(X+n-1)^\alpha} + \frac{1}{(X+n)^\alpha}\right)$.
  Hence, the right hand side of \eqref{EulerMac} is at most 
\begin{equation*}\sum_{n=1}^\infty \frac{1}{2} \left(
  \frac{1}{(X+n-1)^\alpha} - \frac{1}{(X+n)^\alpha}\right) 
  \leq  \frac{1}{2X^\alpha}.
\end{equation*}
On the other hand, by the mean value theorem, for any $n\in\mathbb{Z}_{>0}$, there exists $r\in[n-1,n]$ such that
$\int_{n-1}^{n}\frac{dx}{(X+x)^{\alpha}}-\frac{1}{(\lfloor X\rfloor+n)^{\alpha}}=\frac{1}{(X+r)^{\alpha}}-\frac{1}{(\lfloor X\rfloor+n)^{\alpha}}$. Thus, by the monotonicity of $t\mapsto\frac{1}{t^{\alpha}}$ and the fact that $X+r$ and $\lfloor X\rfloor+n$ are both contained in $[X+n-1,X+n]$, we have that 
the right hand side of \eqref{EulerMac} is at least
\begin{equation*}
%\sum_{n=1}^{\infty}\left(\frac{1}{(X+r)^{\alpha}}-\frac{1}{(\lfloor X\rfloor+n)^{\alpha}}\right)\geq
\sum_{n=1}^{\infty}\left(\frac{1}{(X+n)^{\alpha}}-\frac{1}{(X+n-1)^{\alpha}}\right)=-\frac{1}{X^{\alpha}}.
\end{equation*}
\end{proof}

The following lemma estimates asymptotically some sums even when they have an empty condition.

 \begin{lemma}\label{SumEstimations} 
 Let $X>0$ and $\alpha>0$. If $0<\delta\leq 1$, we have
 \begin{align}\label{harmonic}
 \sum_{n\leq X}\frac{1}{n}&=\log(X)+\gamma+O^*\left(\frac{\Delta_{1}^{\delta}}{X^\delta}\right);
 \end{align}
 if $\max\{0,\alpha-1\}<\delta\leq\alpha$ and $\alpha\neq 1$, we have
 \begin{equation}\label{generalz}
 \sum_{n\leq X}\frac{1}{n^\alpha}=\zeta(\alpha)-\frac{1}{(\alpha-1)X^{\alpha-1}}+O^*\left(\frac{\Delta_{\alpha}^{\delta}}{X^\delta}\right),
 \end{equation} 
where $\Delta_{1}^{\delta}=\max\left\{\gamma,\frac{1}{\delta e^{\gamma\delta+1}}\right\}$
and, for $\alpha\neq 1$,
 \begin{align*}
 \Delta_{\alpha}^{\delta}&=
     \begin{cases}
     \max\left\{1,\left(\frac{1}{\delta^{\delta}}\left(\frac{(\delta-\alpha+1)}{|\zeta(\alpha)(\alpha-1)|}\right)^{\delta-\alpha+1}\right)^{\frac{1}{\alpha-1}},\zeta(\alpha)-\frac{1}{\alpha-1}\right\},&\quad\text{ if }\delta\neq\alpha,\\
     1,&\quad\text{ if }\delta=\alpha.
     \end{cases}
 \end{align*}
 \end{lemma}
  \begin{proof} By \cite[Lemma 2.1]{RA13} and Lemma \ref{recip}, for $X>0$, we have
  \begin{align}
 \sum_{n\leq X}\frac{1}{n}&=\log(X)+\gamma+O^*\left(\frac{\gamma}{X}\right),\label{Ha}\\
 \sum_{n\leq X}\frac{1}{n^\alpha}&=\zeta(\alpha)-\frac{1}{(\alpha-1)X^{\alpha-1}}+O^*\left(\frac{1}{X^\alpha}\right),\quad\text{if\ }\alpha> 0\text{\ and\ }\alpha\neq 1,\label{Ge}
 \end{align}  
  respectively. 
 Thus, if $X\geq 1$, the result holds trivially as $\delta'\mapsto X^{\delta'}$ is increasing and $\delta<\alpha$. Otherwise, when $0<X<1$ the above summations are empty; write $X=\frac{1}{Y}$ with $Y>1$ and           observe first that the function $f:Y\geq 1\mapsto\frac{\log(Y)-\gamma}{Y^\delta}$ has a single critical point at $y_0=e^{\frac{1}{\delta}+\gamma}>1$ taking the value $f(y_0)=\frac{1}{\delta e^{\gamma\delta+1}}>0$. As $f(1)=-\gamma$ and $\lim_{Y\to\infty}f(Y)=0$, $f$ is increasing in $[1,y_0]$ and decreasing in $[y_0,\infty)$, and hence $\sup_{\{Y>1\}}|f(Y)|=\max\left\{\gamma,\frac{1}{\delta e^{\gamma\delta+1}}\right\}$. 

  Secondly, by \cite[Cor. 1.14]{MV07}, we have that $\zeta(\alpha)>\frac{1}{\alpha-1}$ and $\zeta(\alpha)(\alpha-1)>0$ for all $\alpha\geq 0$ and $\alpha\neq 1$. Moreover, the function $g:Y>0\mapsto\frac{1}{Y^\delta}\left(\zeta(\alpha)-\frac{Y^{\alpha-1}}{\alpha-1}\right)$ has a critical point $y_0$ satisfying $y_0^{\alpha-1}=\frac{\zeta(\alpha)(\alpha-1)\delta}{\delta-\alpha+1}>0$, since $\delta>\alpha-1$ and $\delta>0$ and in this case, we have that $\lim_{Y\to\infty}g(Y)=0$ and, thus, $|g|$ is decreasing in $[y_0,\infty)$. We conclude then that $\max_{[y_0,\infty)}|g(Y)|$ $=|g(y_0)|$, where 
 \begin{equation*}
 |g(y_0)|=\left(\frac{1}{\delta^{\delta}}\left(\frac{(\delta-\alpha+1)}{|\zeta(\alpha)(\alpha-1)|}\right)^{\delta-\alpha+1}\right)^{\frac{1}{\alpha-1}}.
 \end{equation*} 
 If $y_0\leq 1$, then $|g(1)|=g(1)\leq|g(y_0)|$ and $\sup_{\{Y>1\}}|g(Y)|=g(1)$; otherwise, if $y_0>1$, as $g$ is also monotonic between $1$ and $y_0$, we derive that $\sup_{\{Y>1\}}|g(Y)|=\max\{g(1),|g(y_0)|\}$, which gives us the desired result.
  \end{proof}

It is important to point out that in case that $\alpha>1$, it would have been possible to give an error term expression even if $\delta=\alpha-1>0$, whereas, if $\delta<\alpha-1$, then $|g|$ would have been unbounded in $[1,\infty)$.

On the other hand, as pointed out at the beginning of \S\ref{dirichlet}, it is essential to have an estimation of the above summations when they have actually an empty condition, that is when $X\in(0,1)$. Indeed, this will provide regularity for some sum conditions during the proof of Theorem \ref{general} that otherwise would impose some variables to be at least $1$ and some sums to be non-empty. It should be expected, though, that the fact of imposing regularity conditions, or rather asking for estimations of sums up to the variable $X$ with $X>0$, will worsen a bit the constants on the involved error terms; for instance, when $\alpha=1$ and when we are restricted to the range $X\geq 1$, the value of $\gamma=\sage{Trunc(gamma,5)}\ldots$ given in \eqref{Ha} can be improved to $2(\log(2)+\gamma-1)=\sage{Trunc(2*(log(2)+euler_gamma-1),5)}\ldots$ (refer to \cite[Lemma 2.1]{RA13} ).

\subsection{The convolution method} \label{TCM}

The following theorem will help us to state Corollary \ref{corollary}. Although inspired by \cite[Lemma 3.2]{RA95}, it is presented in a much general framework, in an attempt to understand and deduce with ease the order of averages of sufficiently regular square-free supported arithmetic functions. By sufficiently regular, we mean an arithmetic function having a specific constant dominant term on all sufficiently large prime numbers. As it turns out, it is precisely the regularity of an arithmetic function that helps to derive the asymptotic expansion of its average under the method of convolution.
 \begin{theorem}\label{general}    
  Let  $q$ a positive integer and let $X$, $\alpha$, $\beta$ be real numbers such that $X>0$, $\beta>1$ and $\beta>\alpha>\frac{1}{2}$. Consider a multiplicative function $f:\mathbb{Z}^+\to\mathbb{C}$ such that 
 $f(p)=\frac{1}{p^{\alpha}}+O\left(\frac{1}{p^{\beta}}\right)$, for every sufficiently large prime number $p$ coprime to $q$. Then for any real number $\delta>0$ such that $\max\{0,\alpha-1\}<\delta<\min\{\beta-1,\alpha-\frac{1}{2}\}$ we have the estimation 
 \begin{equation*}
 \sum_{\substack{\ell\leq X\\ (\ell,q)=1}}\mu^2(\ell)f({\ell})=F_\alpha^{q}(X)+O^*\left(\Delta_{\alpha}^{\delta}\frac{\kappa_{\alpha-\delta}(q)}{q^{\alpha-\delta}}\cdot\frac{\overline{H}_{f}^{\phantom{.}q}(-\delta)}{X^\delta}\right),
 \end{equation*} 
 where, if $\alpha\neq 1$,
 \begin{align*} 
 F_\alpha^q(X)&=\frac{H_{f}^{q}(0)\zeta(\alpha)\varphi_\alpha(q)}{q^\alpha}-\frac{H_{f}^{q}(1-\alpha)\varphi(q)}{(\alpha-1)q}\frac{1}{X^{\alpha-1}},
\end{align*}
and, if $f(p)= -1$ for some prime number $p$,  $F_1^q(X)=-\sum_{d}\frac{h_{f}^{q}(d)\log(d)}{d^\alpha}$, whereas, if $f(p)\neq -1$ for any prime number $p$,
\begin{align*}
 F_1^q(X)&=\frac{H_{f}^{q}(0)\varphi(q)}{q}\left(\log\left(X\right)+T_f^q+\gamma+\sum_{p|q}\frac{\log(p)}{p-1}\right),\\
T_{f}^{q}&=\sum_{p\nmid q}\frac{\log(p)(1-(p-2)f(p))}{(f(p)+1)(p-1)}.
 \end{align*}
 Here, $\Delta_{\alpha}^{\delta}$ is defined as in Lemma \ref{SumEstimations} and  $H_{f}^{q}:\{s\in\mathbb{C},\ \Re(s)>\frac{1}{2}-\alpha\}\to\mathbb{C}$ is an analytic function satisfying  
 \begin{align*}
 H_{f}^{q}(s)&=\prod_{p\nmid q}\left(1-\frac{1-f(p)p^\alpha}{p^{s+\alpha}}-\frac{f(p)}{p^{2s+\alpha}}  \right)=\sum_{\substack{d\\(d,q)=1}}\frac{h_{f}^{q}(d)}{d^{s+\alpha}},\\
 \overline{H}_{f}^{\phantom{.}q}(s)&=\prod_{p\nmid q}\left(1+\frac{|1-f(p)p^\alpha|}{p^{\Re(s)+\alpha}}+\frac{|f(p)|}{p^{2\Re(s)+\alpha}}  \right)=\sum_{\substack{d\\(d,q)=1}}\frac{|h_{f}^{q}(d)|}{d^{\Re(s)+\alpha}}.
 \end{align*} 
 \end{theorem}

  \begin{proof} 
  By the asymptotic condition on $f$ in the statement, the Dirichlet series $D_{f}^{q}$ associated with $\ell\mapsto\mu^2(\ell)f({\ell})\mathds{1}_q(\ell)$, where $\mathds{1}_q$ is defined as the multiplicative function $\ell\mapsto\mathds{1}_{\{(\ell,q)=1\}}(\ell)$, converges absolutely for any $s\in\mathbb{C}$ such that $\Re(s)>1-\alpha$. Thus, in the set $\{s\in\mathbb{C},  \ \Re(s)>1-  \alpha\}$, the equality  
  \begin{align}
  D_{f}^{q}(s)=\sum_{\substack{\ell\\(\ell,q)=1}}\frac{\mu^2(\ell)f(\ell)}{\ell^{s}}=\prod_{p\nmid q}\left(1+\frac{f(p)}{p^s}\right)
  \end{align}
  holds and the function $s\mapsto\zeta(s+\alpha)$ can be expressed by an Euler product.  For any $s$ such that $\Re(s)>1-\alpha$, we have then
  \begin{align}\label{D}
  &\frac{D_{f}^{q}(s)}{\zeta(s+\alpha)}=\prod_{p\nmid q}\left(1+\frac{f(p)}{p^s}\right)\left(1-\frac{1}{p^{s+\alpha}}\right)\cdot\prod_{p|q}\left(1-\frac{1}{p^{s+\alpha}}\right)\nonumber\\
  &\phantom{xxxx}=\frac{\varphi_{s+\alpha}(q)}{q^{s+\alpha}}\cdot\prod_{p\nmid q}\left(1-\frac{1-f(p)p^\alpha}{p^{s+\alpha}}-\frac{f(p)}{p^{2s+\alpha}}\right)\ \ =\ \frac{\varphi_{s+\alpha}(q)}{q^{s+\alpha}}\cdot H_{f}^{q}(s).\nonumber
  \end{align}
  Also, we have that $\frac{1-f(p)p^{\alpha}}{p^{s+\alpha}}=O\left(\frac{1}{p^{\Re(s)+\beta}}\right)$ and $\frac{f(p)}{p^{2s+\alpha}}=O\left(\frac{1}{p^{2\Re(s)+2\alpha}}\right)$. Since $\beta>\alpha$, we have that $H$ can be extended analytically from $\{s\in\mathbb{C},\ \Re(s)>1-\alpha\}$ onto $\{s\in\mathbb{C},\ \Re(s)>\max\{1-\beta,\frac{1}{2}-\alpha\}\}$. Further, as $0>1-\beta$ and $0>\frac{1}{2}-\alpha$, $H_{f}^{q}(0)$ exists and, if $f(p)\neq -1$ for any prime number $p$, it is different from $0$, since each factor defining it can be expressed as $(1+f(p))\left(1-\frac{1}{p^\alpha}\right)$ and $\alpha\neq 0$.

  Now, the formal equality $D_{f}^{q}(s)=H_{f}^{q}(s)\cdot\prod_{p\nmid q}\left(1+\frac{1}{p^{s+\alpha}}+\frac{1}{p^{2(s+\alpha)}}+\ldots\right)$ hides the convolution product
  \begin{equation}\label{identity1}
  \ell^\alpha \mu^2(\ell)f(\ell)\mathds{1}_{(\ell,q)=1}(\ell)=(h_{f}^{q}\star\mathds{1}_q)\ (\ell)=\sum_{\substack{d|\ell}}h_{f}^{q}(d)\mathds{1}_q\left(\frac{\ell}{d}\right), 
  \end{equation}
  where $h$ is a multiplicative function defined on the prime numbers as
  \begin{align}
  &h_{f}^{q}(p)=(f(p)p^\alpha-1)\cdot\mathds{1}_q(p),\qquad h_{f}^{q}(p^2) = -f(p)p^\alpha\cdot \mathds{1}_q(p),\label{hfq}\\
  &\phantom{xxxxxxxxxxxxxx} h_{f}^{q}(p^k) = 0, \quad k>2.\nonumber
  \end{align}
  Therefore, from \eqref{identity1} we conclude that
  \begin{align}\label{DirichletIdentity}
  \sum_{\substack{\ell\leq X\\ (\ell,q)=1}}\mu^2(\ell)f({\ell})=\sum_{\substack{\ell\leq X}}\frac{(h_{f}^{q}\star\mathds{1}_q)\ (\ell)}{\ell^\alpha}=\sum_{\substack{d}}\frac{h_{f}^{q}(d)}{d^\alpha}    \sum_{\substack{e\leq\frac{X}{d}\\(e,q)=1}}  \frac{1}{e^\alpha}\phantom{xxxxxxxxxxxxx}&\nonumber\\
=\sum_{\substack{d}}\frac{h_{f}^{q}(d)}{d^\alpha}\sum_{\substack{e\leq\frac{X}{d}}}\frac{1}{e^\alpha}\sum_{d'|e,d'|q}\mu(d')=\sum_{\substack{d}}\frac{h_{f}^{q}(d)}{d^\alpha}\sum_{d'|q}\frac{\mu(d')}{d'^\alpha}\sum_{\substack{e\leq\frac{X}{dd'}}}\frac{1}{e^\alpha}&,
  \end{align}
  where there is no upper bound conditions on the variables $d$ and $d'$ present in the outer sums above, their being encoded by the innermost  sum of  \eqref{DirichletIdentity}, which, in order to continue our analysis, we must estimate regardless of whether or not it is empty: Lemma \ref{SumEstimations} allow us to handle this situation. 

Hence, as $\max\{0,\alpha-1\}<\delta<\min\{\beta-1,\alpha-\frac{1}{2}\}<\alpha$, we derive that the second sum in \eqref{DirichletIdentity} can be expressed as
  \begin{align}\label{Sum:alpha neq 1}
  \sum_{d'|q}\frac{\mu(d')}{d'^\alpha}\sum_{\substack{e\leq\frac{X}{dd'}}}\frac{1}{e^\alpha}=\sum_{d'|q}\frac{\mu(d')}{d'^\alpha}\left(\zeta(\alpha)-\frac{(dd')^{\alpha-1}}  {(\alpha-1)X^{\alpha-1}}+O^*\left(\Delta_{\alpha}^{\delta}\frac{(dd')^\delta}{   X^{\delta}}\right)\right)&\nonumber\\
 =\ \frac{\zeta(\alpha)\varphi_\alpha(q)}{q^\alpha}-\frac{\varphi(q)}{(\alpha-1)q}\cdot\frac{d^{\alpha-1}}{X^{\alpha-1}}+O^*\left(\Delta_{\alpha}^{\delta}\frac{\kappa_{\alpha-\delta}(q)}{q^{\alpha-\delta}}\cdot\frac{d^\delta}{X^\delta}\right)&,
  \end{align}
  if $\alpha\neq 1$, or as
  \begin{align}\label{Sum:alpha=1}
  \sum_{d'|q}\frac{\mu(d')}{d'^\alpha}\sum_{\substack{e\leq\frac{X}{dd'}}}\frac{1}{e^\alpha}=\sum_{d'|q}\frac{\mu(d')}{d'^\alpha}\left(\log\left(\frac{X}{dd'}\right)+  \gamma+O^*\left(\frac{\Delta_1^{\delta}(dd')^\delta}{X^\delta}\right)\right)\phantom{xxxxxx}&  \nonumber\\
  =\frac{\varphi_\alpha(q)}{q^\alpha}\left(\log\left(\frac{X}{d}\right)+\gamma\right)-\sum_{d'|q}\frac{\mu(d')\log(d')}{d'^\alpha}+O^*\left(\frac{\Delta_1^{\delta}\kappa_{\alpha-\delta}(q)}{q^{\alpha-\delta}}\cdot\frac{d^\delta}  {X^\delta}\right)&\nonumber\\
  =\frac{\varphi_\alpha(q)}{q^\alpha}\left(\log\left(\frac{X}{d}\right)+\gamma+\sum_{p|q}\frac{\log(p)}{p^\alpha-1}\right)+O^*\left(\frac{\Delta_1^{\delta}\kappa_{\alpha-\delta}(q)}{q^{\alpha-\delta}}\cdot\frac{d^\delta}{X^\delta}  \right),&
  \end{align}
  if $\alpha=1$, where we have used that
  \begin{align}-\sum_{d'|q}\frac{\mu(d')\log(d')}{d'^\alpha}=\left(\frac{\varphi_{s+\alpha}(q)}{q^{s+\alpha}}\right)'_{s=0}=\frac{\varphi_\alpha(q)}{q^\alpha}\sum_{p|q}  \frac{\log(p)}{p^\alpha-1}.\label{derivated}
  \end{align}
  On the other hand, observe that $H_{f}^{q}(1-\alpha)$ and $\overline{H}_{f}^{\phantom{.}q}(-\delta)$ are well-defined, as $\min\{1-\alpha,-\delta\}>\max\{1-\beta,\frac{1}{2}-\alpha\}$. Therefore, from \eqref{DirichletIdentity}, the sum $\sum_{\substack{\ell\leq X\\ (\ell,q)=1}}\mu^2(\ell)f({\ell})$ can be estimated either as
  \begin{align}\label{DirichletIdentity:alpha=1} 
  \sum_{\substack{d}}\frac{h_{f}^{q}(d)}{d^\alpha}\left(\frac{\zeta(\alpha)\varphi_\alpha(q)}{q^\alpha}-\frac{\varphi(q)}{(\alpha-1)q}\cdot\frac{d^{\alpha-1}}{X^{\alpha-1}}  +O^*\left(\frac{\Delta_{\alpha}^{\delta}\kappa_{\alpha-\delta}(q)}{q^{\alpha-\delta}}  \cdot\frac{d^\delta}{X^\delta}\right)\right)&\\
  =H_{f}^{q}(0)\ \frac{\zeta(\alpha)\varphi_\alpha(q)}{q^\alpha}-\frac{\varphi(q)}{(\alpha-1)q}\cdot\frac{H_{f}^{q}(1-\alpha)}{X^{\alpha-1}}+O^*\left(\frac{\Delta_{\alpha}^{\delta}\kappa_{\alpha-\delta}(q)}{q^{\alpha-\delta}}\cdot\frac{\overline{H}_{f}^{\phantom{.}q}(-\delta)}{X^\delta}  \right)\nonumber&,
  \end{align}
  if $\alpha\neq 1$, by using \eqref{Sum:alpha neq 1}, or 
  \begin{align}\label{DirichletIdentity:alphaneq1}
  \sum_{\substack{d}}\frac{h_{f}^{q}(d)}{d^\alpha}\left(\frac{\varphi_\alpha(q)}{q^\alpha}\left(\log\left(\frac{X}{d}\right)+\gamma+\sum_{p|q}\frac{\log(p)}{p^\alpha-1}\right)+O^*\left(\frac{\Delta_{1}^{\delta}\kappa_{\alpha-\delta}(q)}{q^{\alpha-\delta}}\cdot\frac{d^\delta}{X^\delta}\right)\right)&\\
  = H_{f}^{q}(0)\ \frac{\varphi_\alpha(q)}{q^\alpha}\left(\log\left(X\right)+\gamma+\sum_{p|q}\frac{\log(p)}{p^\alpha-1}\right)+H_{f}^{q}\phantom{}'(0)\phantom{xx}\nonumber&\\
 + O^*\left(\frac{\Delta_{1}^{\delta}\kappa_{\alpha-\delta}(q)}{q^{\alpha-\delta}}  \cdot\frac{\overline{H}_{f}^{\phantom{.}q}(-\delta)}{X^\delta}\right),\phantom{xx}\nonumber&
  \end{align}
  if $\alpha=1$, by using \eqref{Sum:alpha=1} and that $-\sum_{d}\frac{h_{f}^{q}(d)\log(d)}{d^\alpha}=H_{f}^{q}\phantom{}'(0)$. The result is thus obtained by noticing that if $H_{f}^{q}(0)\neq 0$, then $\frac{H_{f}^{q}\phantom{}'(0)}{H_{f}^{q}(0)}$ equals
  \begin{align*}
  \left(\prod_{p\nmid q}\left(1-\frac{1-f(p)p^\alpha}{p^{s+\alpha}}-\frac{f(p)}{p^{2s+\alpha}}\right)\right)'_{s=0}=\sum_{p\nmid q}\frac{\log(p)(1-f(p)p^\alpha+2f(p))}{(f(p)+1)(p^\alpha-1)}.
  \end{align*} 
  \end{proof}

 \begin{corollary}\label{corollary}

 Let $X>0$ and $q\in\mathbb{Z}_{>0}$. The following estimations hold
 \begin{align} 
 \mathbf{(a)}\sum_{\substack{\ell\leq X\\ (\ell,q)=1}}\frac{\mu^2(\ell)}{\varphi(\ell)}&=\frac{\varphi(q)}{q}\left(\log\left(X\right)+\mathfrak{a}_q\right)+O^*\left(\frac{     \sage{Error_sum1}     \cdot\mathpzc{A}_q}{X^{\sage{delta}}}\right)\label{sum1:eq}, \\ 
 \mathbf{(b)} \sum_{\substack{\ell\leq X\\ (\ell,q)=1}}\frac{\mu^2(\ell)}{\ell}&=\frac{6}{\pi^2}\frac{q}{\kappa(q)}\left(\log\left(X\right)+\mathfrak{b}_q\right)+O^*\left(\frac{\sage{Upper(Error_sum2,digits)}  \cdot\mathpzc{B}_q}{X^{\sage{delta}}}\right),\label{sum2:eq}
  \end{align}  
 where
 \begin{align*}
  \mathpzc{A}_q= \prod_{p|q}\left(1+\frac{p-p^{\sage{delta}}-2}{(p-1)p^{\sage{1-delta}}+p^{\sage{delta}}+1}\right),\mathpzc{B}_q =\prod_{p|q}\left(1+\frac{p^{\sage{1-delta}}-1}{p^{\sage{2-2*delta}}+1}\right),
  \end{align*} 
and
 \begin{align*} 
 \mathfrak{a}_q&=\sum_{p}\frac{\log(p)}{p(p-1)}+\gamma+\sum_{p|q}\frac{\log(p)}{p},\sum_{p}\frac{\log(p)}{p(p-1)}+\gamma=   \sage{C_sum1}    \ldots,\\
\mathfrak{b}_q &=\sum_{p}\frac{2\log(p)}{p^2-1}+\gamma+\sum_{p|q}\frac{\log(p)}{p+1}, 
\sum_{p}\frac{2\log(p)}{p^2-1}+\gamma=\sage{C_sum2}\ldots .
\end{align*} 
 \end{corollary}

  \begin{proof}
  For the case $\mathbf{(a)}$ (respectively $\mathbf{(b)}$), apply Theorem \ref{general} with $f(p)=\frac{1}{\varphi(p)}=\frac{1}{p-1}$ (respectively $f(p)=\frac{1}{p}$), $\alpha=1$, $\beta=2$ and $0\leq\delta=\sage{delta}<\frac{1}{2}$. 

The infinite products that participate in the main and error terms as well as the infinite summation that participates in the main term can be estimated by using a rigorous implementation of interval arithmetic, and some techniques for accelerating convergence.
  \end{proof}

\noindent\textbf{Remarks.} Conditions $\alpha>\frac{1}{2}$ and $\beta>1$ in Theorem \ref{general} are necessary to ensure the existence of $H_{f}^{q}(0)$. Nonetheless, we can derive an analogous result for any multiplicative arithmetic function $f$ satisfying the conditions $f(p)=\frac{1}{p^{\alpha}}+O\left(\frac{1}{p^{\beta}}\right)$, for every sufficiently large prime number $p$ coprime to $q$, where $\alpha\leq\frac{1}{2}$ and $\beta>\alpha$ by using of Theorem \ref{general} and summation by parts. In this instance, there will not be any secondary term appearing and the error term magnitude will be $O\left(X^{1-\alpha-\delta}\right)$ for any $0<\delta<\min\{\beta-\alpha,\frac{1}{2}\}$

Upon having Theorem \ref{general} at our disposal, the asymptotic estimation of averages $\sum_{\substack{\ell\leq X\\ (\ell,q)=1}}\mu^2(\ell)f({\ell})$ satisfying conditions of that theorem becomes an automatized, but not uninteresting task, that involves each time a choice of parameters: a value for $\delta$ and a precision value in order to obtain a rigorous estimation of some infinite products.

In general, we have freedom to choose the error term parameter $\delta$ described in \S\ref{dirichlet} but some of them are not optimal. For instance, if $\alpha=1$, then in terms of Theorem \ref{general} and Lemma \ref{SumEstimations}, $\Delta_1^\delta\to\infty$ as $\delta\to 0^{+}$. Since $\overline{H}_f^{\phantom{.}q}(-\delta)$ converges, that makes the expression $\Delta_1^\delta\overline{H}_f^{\phantom{.}q}(-\delta)$ tending to $\infty$ as well, thus not providing a numerical acceptable value. On the other hand,  when $\delta\to\frac{1}{2}^-$, the infinite product given by $\overline{H}_f^{\phantom{.}q}(-\delta)$ tends to $\infty$, whereas $\Delta_1^\delta\to\Delta_1^{\frac{1}{2}}$, thus bounded, so that one also derives that the expression $\Delta_1^\delta\overline{H}_f^{\phantom{.}q}(-\delta)$ becomes too big to be practical. The search looks for a value of $\delta$ not too close to the boundaries of $(0,\frac{1}{2})$, and in almost all cases it seems acceptable to set $\delta=\frac{1}{3}$. 

A natural question is whether or not we can improve on the error term estimation given in Theorem \ref{general}, mandatorily with a different method, of exponent $\delta=\min\{\beta-1,\alpha-\frac{1}{2}\}$.  
When $\beta-\alpha>\frac{1}{2}$, then $\delta=\alpha-\frac{1}{2}$ and the answer to that question is given in \S\ref{improvement}: it is positive and constitutes our main result. We provide in addition explicit estimations for those \textit{ critical exponents }.

Out of the results above, the sum \eqref{sum1:eq} is classical and it has been thoroughly studied by Ramar\'e and Akhilesh in \cite{RA13}, by Ramar\'e in \cite[Thm. 3.1]{RA19}, \cite[Lemma 3.4]{RA95} and given in our simpler form by Helfgott in \cite[\S 6.1.1]{Hel19}.

\section{Improvements on the convolution method}\label{improvement}

During the proof of Theorem \ref{general}, it was crucial to have an empty sum estimation for the inner sum given in \eqref{DirichletIdentity} so that, thanks to the regularity on the variable $d$ we find convergent main and error term coefficients, as shown in \eqref{DirichletIdentity:alpha=1} and \eqref{DirichletIdentity:alphaneq1}.  

This general idea misses the fact that the function $h_f^q$ defined in \eqref{hfq} vanishes on all non cube-free numbers, and that the particular function $h_f^q:p,(p,q)=1\mapsto\frac{1}{p^\alpha}$, with $\alpha>\frac{1}{2}$, satisfies $h_f^q(p)=0$. Moreover, the fact that that particular function is meaningful only on the square of the prime numbers, will allow us to achieve the critical exponent $\delta=\frac{1}{2}$, if $\alpha=1$ or $\delta=\alpha-\frac{1}{2}$, if $\alpha\neq 1$ and $\alpha>\frac{1}{2}$, when $f$ is an arithmetic function satisfying the conditions of Theorem \ref{general} with $\beta-\alpha>\frac{1}{2}$.

\subsection{A particular case} \label{particular}

Let us see how we can improve the estimation $\mathbf{(b)}$ given in Corollary \ref{corollary}.
\begin{lemma}\label{sum2:critic:1} 
Let $X>0$.
Then
\begin{align}\label{sum2:v1} 
\sum_{\ell\leq X}\frac{\mu^2(\ell)}{\ell}&=\frac{6}{\pi^2}\left(\log(X)+\mathfrak{b}_1\right)+O^*\left(\frac{\sage{Upper(CONSTANT1,digits)}}{\sqrt{X}}\right),\\
\label{sum2:v2}\sum_{\substack{\ell\leq X\\(\ell,2)=1}}\frac{\mu^2(\ell)}{\ell}&=\frac{4}{\pi^2}\left(\log(X)+\mathfrak{b}_2\right)+O^*\left(\frac{\sage{Upper(CONSTANT_ALT,digits)}}{\sqrt{X}}\right),
\end{align}  
where $\mathfrak{b}_1=\gamma+\sum_{p}\frac{2\log(p)}{p^2-1}=\sage{C_sum2}\ldots$, $\mathfrak{b}_2=\mathfrak{b}_1+\frac{\log(2)}{3}=\sage{C_sum2_22}\ldots$. 

If we restraint ourselves to the range $X\geq 1$,
then $\sage{Upper(CONSTANT1,digits)}$ may be replaced by $\sage{Upper(Ram_new_cst,2)}$ and $\sage{Upper(CONSTANT_ALT,digits)}$ may be replaced by $\sage{Upper(max(C2_v2_1,C2_v2_2),digits)}$.
\end{lemma} 

\begin{proof} Equation \eqref{sum2:eq} gives the  main term of \eqref{sum2:v2} and from that, we can conclude by summation by parts that for all $X\geq 1$, $\sum_{\substack{\ell\leq X\\(\ell,2)=1}}\frac{\mu^2(\ell)}{\ell}$ equals
\begin{align}\label{eqq:v=2}
\frac{4(\log(X)+\mathfrak{b}_2)}{\pi^2}+\left(\sum_{\substack{\ell\leq X\\(\ell,2)=1}}\mu^2(\ell)-\frac{4}{\pi^2}X\right)\frac{1}{X}-\int_X^{\infty}\left(\sum_{\substack{\ell\leq t\\(\ell,2)=1}}\mu^2(\ell)-\frac{4}{\pi^2}t\right)\frac{dt}{t^2}.
\end{align}
Moreover, by \cite[Lemma 5.2]{Hel19}, we have 
\begin{align}
\sup_{\{X\geq 1573\}}\frac{1}{\sqrt{X}}\left|\sum_{\substack{\ell\leq X\\(\ell,2)=1}}\mu^2(\ell)-\frac{4}{\pi^2}X\right|&\leq\frac{9}{70}\label{bound:1},
\end{align}
so that, by \eqref{eqq:v=2}, 
\begin{align*}
\sum_{\substack{\ell\leq X\\(\ell,2)=1}}\frac{\mu^2(\ell)}{\ell}=\frac{4}{\pi^2}(\log(X)+\mathfrak{b}_2)+O^*\left(\frac{27}{70}\frac{1}{\sqrt{X}}\right),\quad\text{ if }X\geq 1573,
\end{align*}
where $\frac{27}{70}=\sage{Trunc(C2_v2_1,digits)}\ldots$. We further verify by interval arithmetic that
\begin{align}\label{intt:v=2}
\sup_{\{1\leq X\leq 1573\}}\sqrt{X}\left|\sum_{\substack{\ell\leq X\\(\ell,2)=1}}\frac{\mu^2(\ell)}{\ell}-\frac{4}{\pi^2}(\log(X)+\mathfrak{b}_2)\right|\leq\sage{Upper(C2_v2_2,digits)}
\end{align} 
the above upper bound being almost achieved when $X\to 3^-$. On the other hand \cite[Cor. 1.2]{RA19} tells us that
\begin{align}
\sup_{\{X\geq 1\}}\sqrt{X}\left|\sum_{\ell\leq X}\frac{\mu^2(\ell)}{\ell}-\frac{6}{\pi^2}\left(\log(X)+\mathfrak{b}_1\right)\right|&\leq\sage{Upper(Ram_new_cst,2)}\label{sum2:bound:X>=1}.
\end{align}
Hence, by using \eqref{bound:1}, \eqref{intt:v=2} and \eqref{sum2:bound:X>=1}, when $v\in\{1,2\}$, we have the bounds
\begin{align}\label{BB}
\sup_{\{X\geq 1\}}\sqrt{X}\left|\sum_{\substack{\ell\leq X\\(\ell,v)=1}}\frac{\mu^2(\ell)}{\ell}-\frac{v}{\kappa(v)}\frac{6}{\pi^2}(\log(X)+\mathfrak{b}_v)\right|\leq
\begin{cases}
\sage{Upper(Ram_new_cst,2)}&,\quad\text{ if }v=1,\\
\sage{Upper(max(C2_v2_1,C2_v2_2),digits)} &,\quad\text{ if }v=2.
\end{cases}&
\end{align} 
In order to derive the result, it is sufficient to obtain bounds for \eqref{BB} when $X\in(0,1)$, in which case the above summation vanishes. By defining $Y=\frac{1}{X}>1$ and $t_v:Y\mapsto\frac{6v(\log(Y)-\mathfrak{b}_v)}{\kappa(v)\pi^2\sqrt{Y}}$, we need to find $\sup_{\{Y>1\}}|t_v(Y)|$. By calculus, the function $t_v$ has a critical point at $y_0=e^{2+\mathfrak{b}_v}$, with value $t_v(y_0)=\frac{12v}{\kappa(v)\pi^2e^{1+\frac{\mathfrak{b}_v}{2}}}$, and it is monotonic in $[1,y_0]$ and in $[y_0,\infty)$. As $\lim_{Y\to \infty}t_v(Y)=0$ and $t_v(y_0)>0$, we conclude that $t_v$ is decreasing in $[y_0,\infty)$. Similarly, as $t_v(1)=-\frac{6v\mathfrak{b}_v}{\kappa(v)\pi^2}<0$, $t_v$ is increasing in $[1,y_0]$. Therefore 
\begin{align}\label{BB(0,1)}
\sup_{\{0<X<1\}}\sqrt{X}\left|\sum_{\substack{\ell\leq X\\(\ell,v)=1}}\frac{\mu^2(\ell)}{\ell}-\frac{v}{\kappa(v)}\frac{6}{\pi^2}(\log(X)+\mathfrak{b}_v)\right|=\max\{|t_v(1)|,|t_v(y_0)|\}\nonumber&\\
=\frac{6v\mathfrak{b}_v}{\kappa(v)\pi^2}=\begin{cases}
\sage{Upper(6*C_sum2/pi^2,digits)}, &\quad\text{ if }v=1,\\
\sage{Upper(4*C_sum2_22/pi^2,digits)},&\quad\text{ if }v=2.
\end{cases}&
\end{align} 
Finally, whenever either $v=1$ or $v=2$, the constant in the error term is obtained by taking the maximum between the bounds \eqref{BB} and \eqref{BB(0,1)}.
\end{proof}

\begin{lemma}\label{sum2:critic}   
Let $X>0$ and $\alpha>\frac{1}{2}$. If $\alpha\neq 1$, then
\begin{align*}
\sum_{\ell\leq X}\frac{\mu^2(\ell)}{\ell^\alpha}&=\frac{\zeta(\alpha)}{\zeta(2\alpha)}-\frac{6}{(\alpha-1)\pi^2}\frac{1}{X^{\alpha-1}}+O^*\left(\frac{\mathrm{E}_\alpha^{(1)}}{X^{\alpha-\frac{1}{2}}}\right),\\
\sum_{\substack{\ell\leq X\\(\ell,2)=1}}\frac{\mu^2(\ell)}{\ell^\alpha}&=\frac{2^\alpha}{(2^{\alpha}+1)}\frac{\zeta(\alpha)}{\zeta(2\alpha)}-\frac{4}{(\alpha-1)\pi^2}\frac{1}{X^{\alpha-1}}+O^*\left(\frac{\sqrt{2}}{\varphi_{\frac{1}{2}}(2)}\frac{\mathrm{E}_\alpha^{(2)}}{X^{\alpha-\frac{1}{2}}}\right),
\end{align*}
where, for $v\in\{1,2\}$, we have
\begin{align*}
\mathrm{E}_\alpha^{(v)}=\max\left\{\mathrm{D}_v\left(1+\frac{|\alpha-1|}{\alpha-\frac{1}{2}}\right),\frac{\varphi_{\frac{1}{2}}(v)}{\sqrt{v}}\left|\frac{v^\alpha}{\kappa_\alpha(v)}\frac{\zeta(\alpha)}{\zeta(2\alpha)}-\frac{v}{\kappa(v)}\frac{6}{(\alpha-1)\pi^2}\right|,\phantom{xxx}\right.&\\
\left.\frac{\varphi_{\frac{1}{2}}(v)}{\sqrt{v}}\frac{|\alpha-1|}{\alpha-\frac{1}{2}}\left(\frac{3\kappa_\alpha(v)\zeta(2\alpha)}{\left(\alpha-\frac{1}{2}\right)v^{\alpha-1}\kappa(v)\pi^2|\zeta(\alpha)(\alpha-1)|}\right)^{\frac{2}{\alpha-1}}\right\}\phantom{x}&
\end{align*}
and $\mathrm{D}_1=\sage{Upper(Ram_new_cst,2)}$, $\mathrm{D}_2=\sage{Upper((1-1/sqrt(2))*max(C2_v2_1,C2_v2_2),digits)}.$

If $X\geq 1$, and $\alpha\neq 1$ then we can replace $\mathrm{E}_\alpha^{(v)}$ by $\mathrm{D}_v\left(1+\frac{|\alpha-1|}{\alpha-\frac{1}{2}}\right)$. 
\end{lemma}

\begin{proof} 
If $X\geq 1$, by summation by parts, we can write $\sum_{\substack{\ell\leq X\\(\ell,v)=1}}\frac{\mu^2(\ell)}{\ell^\alpha}$ as
\begin{align}
&\left(\sum_{\substack{\ell\leq X\\(\ell,v)=1}}\frac{\mu^2(\ell)}{\ell}-\frac{v}{\kappa(v)}\frac{6\left(\log(X)+\mathfrak{b}_v\right)}{\pi^2}\right)\frac{1}{X^{\alpha-1}}-\frac{v}{\kappa(v)}\frac{6}{(\alpha-1)\pi^2}\frac{1}{X^{\alpha-1}}+\nonumber\\
&\frac{v}{\kappa(v)}\frac{6(\mathfrak{b}_v(\alpha-1)+1)}{\pi^2(\alpha-1)}+(\alpha-1)\int_1^X\left(\sum_{\substack{\ell\leq t\\(\ell,v)=1}}\frac{\mu^2(\ell)}{\ell}-\frac{v}{\kappa(v)}\frac{6\left(\log(t)+\mathfrak{b}_v\right)}{\pi^2}\right)\frac{dt}{t^{\alpha}}.\nonumber\\\label{step:alpha}
\end{align}
By Theorem \ref{general}, when $\alpha>\frac{1}{2}$, the main term in the asymptotic expression of the above summation is $\frac{v^\alpha}{\kappa_\alpha(v)}\frac{\zeta(\alpha)}{\zeta(2\alpha)}-\frac{v}{\kappa(v)}\frac{6}{
(\alpha-1)\pi^2}\frac{1}{X^{\alpha-1}}
$. By using Lemma \ref{sum2:critic:1} and by making $X\to\infty$, we conclude from  \eqref{step:alpha} that $\frac{v}{\kappa(v)}\frac{6(\mathfrak{b}(\alpha-1)+1)}{\pi^2(\alpha-1)}+(\alpha-1)\int_1^\infty\left(\sum_{\substack{\ell\leq t\\(\ell,v)=1}}\frac{\mu^2(\ell)}{\ell}-\frac{6}{\pi^2}\left(\log(t)+\mathfrak{b}_v\right)\right)\frac{dt}{t^{\alpha}}=\frac{v^\alpha}{\kappa_\alpha(v)}\frac{\zeta(\alpha)}{\zeta(2\alpha)}$. Further, by equation \eqref{BB}, we conclude that, for all $X\geq 1$, $\sum_{\substack{\ell\leq X\\(\ell,v)=1}}\frac{\mu^2(\ell)}{\ell^\alpha}$ is equal to
\begin{align*}
\frac{v^\alpha}{\kappa_\alpha(v)}\frac{\zeta(\alpha)}{\zeta(2\alpha)}-\frac{v}{\kappa(v)}\frac{6}{(\alpha-1)\pi^2}\frac{1}{X^{\alpha-1}}+O^*\left(\frac{\sqrt{v}\ \mathrm{D}_v}{\varphi_{\frac{1}{2}}(v)}\left(1+\frac{|\alpha-1|}{\alpha-\frac{1}{2}}\right)\frac{1}{X^{\alpha-\frac{1}{2}}}\right), 
\end{align*}
where $\mathrm{D}_1=\sage{Upper(Ram_new_cst,2)}$ and $\frac{\varphi_{\frac{1}{2}}(2)}{\sqrt{2}}\sage{Upper(max(C2_v2_1,C2_v2_2),digits)}\leq\mathrm{D}_2=\sage{Upper((1-1/sqrt(2))*max(C2_v2_1,C2_v2_2),digits)}$. 

Suppose now that $X\in(0,1)$. Define $g:X>0\mapsto\frac{v^{\alpha-1}\kappa(v)\pi^2\zeta(\alpha)(\alpha-1)}{6\kappa_\alpha(v)\zeta(2\alpha)}X^{\alpha-\frac{1}{2}}$ $-\sqrt{X}$. We have by \cite[Cor. 1.14]{MV07} that $1<\zeta(\alpha)(\alpha-1)<\alpha$. If $\alpha>1$, we derive that $\frac{\zeta(\alpha)(\alpha-1)}{\zeta(2\alpha)}>\frac{1}{\zeta(2)}$. As $\frac{v^{\alpha-1}\kappa(v)}{\kappa_\alpha(v)}=\frac{1+\frac{1}{v}}{1+\frac{1}{v^\alpha}}>1$ we conclude that $g(1)>0$ and $g$ has a critical point $x_0$ satisfying $0<x_0^{\alpha-1}=\frac{3\kappa_\alpha(v)\zeta(2\alpha)}{\left(\alpha-\frac{1}{2}\right)v^{\alpha-1}\kappa(v)\pi^2|\zeta(\alpha)(\alpha-1)|}<1$, with value $g(x_0)=\frac{1-\alpha}{\alpha-\frac{1}{2}}\sqrt{x_0}<0$. As $g(0)=0$, we conclude that if $\alpha>1$, then $\sup_{\{0<X<1\}}|g(X)|=\max\{g(1),|g(x_0)|\}$.

On the other hand, if $\frac{1}{2}<\alpha<1$, then $2\alpha-1<1$, $\zeta(\alpha)(\alpha-1)<\alpha<1<\frac{\zeta(2\alpha)}{\zeta(2)}$ and $\frac{v^{\alpha-1}\kappa(v)}{\kappa_\alpha(v)}<1$, so that $g(1)<0$. Moreover, the critical point $x_0$ of $g$ satisfies $x_0^{1-\alpha}<1$, so that $x_0<1$, and $g(x_0)>0$. Therefore, if $\frac{1}{2}<\alpha<1$, then $\sup_{\{0<X<1\}}|g(X)|=\max\{|g(1)|,g(x_0)\}$.

All in all, we derive
\begin{align}
\sup_{\{0<X<1\}}X^{\alpha-\frac{1}{2}}\left|\sum_{\substack{\ell\leq X\\(\ell,v)=1}}\frac{\mu^2(\ell)}{\ell^\alpha}-\frac{v^\alpha}{\kappa_\alpha(v)}\frac{\zeta(\alpha)}{\zeta(2\alpha)}+\frac{v}{\kappa(v)}\frac{6}{(\alpha-1)\pi^2}\frac{1}{X^{\alpha-1}}\right|&\nonumber\\
=\frac{v}{\kappa(v)}\frac{6}{|\alpha-1|\pi^2}\max\{|g(1)|,|g(x_0)|\},\label{boundAlpha}&
\end{align}
where,
\begin{align*}
&\frac{\varphi_{\frac{1}{2}}(v)}{\sqrt{v}}\frac{v}{\kappa(v)}\frac{6|g(1)|}{|\alpha-1|\pi^2}=\frac{\varphi_{\frac{1}{2}}(v)}{\sqrt{v}}\left|\frac{v^\alpha}{\kappa_\alpha(v)}\frac{\zeta(\alpha)}{\zeta(2\alpha)}-\frac{v}{\kappa(v)}\frac{6}{(\alpha-1)\pi^2}\right|,\\
&\frac{\varphi_{\frac{1}{2}}(v)}{\sqrt{v}}\frac{v}{\kappa(v)}\frac{6|g(x_0)|}{|\alpha-1|\pi^2}=\\
&\phantom{xxxxxxxxxxxxxx}\frac{\varphi_{\frac{1}{2}}(v)}{\sqrt{v}}\frac{|\alpha-1|}{\alpha-\frac{1}{2}}\left(\frac{3\kappa_\alpha(v)\zeta(2\alpha)}{\left(\alpha-\frac{1}{2}\right)v^{\alpha-1}\kappa(v)\pi^2|\zeta(\alpha)(\alpha-1)|}\right)^{\frac{2}{\alpha-1}}.\nonumber
\end{align*}
The result is obtained by defining $\mathrm{E}_\alpha^{(v)}, v\in\{1,2\}$, as the maximum between $\mathrm{D}_v\left(1+\frac{|\alpha-1|}{\alpha-\frac{1}{2}}\right)$ and the expression \eqref{boundAlpha}.
\end{proof} 
 
\begin{remark}\label{newdef} With the goal of obtaining a similar error shape in Lemma \ref{sum2:critic:1} to the one given in Lemma \ref{sum2:critic}, we extend the definition of $\mathrm{E}_\alpha^{(v)}$, $v\in\{1,2\}$ for $\alpha>\frac{1}{2}$, 
$\alpha\neq 1$, to the case $\alpha=1$ by defining $\mathrm{E}_1^{(1)}=\sage{Upper(CONSTANT1,digits)}$ and, upon observing that $\frac{\varphi_{\frac{1}{2}}(2)}{\sqrt{2}}\sage{Upper(4*C_sum2_22/pi^2,digits)}\leq\sage{Upper((1-1/sqrt(2)) * 4*C_sum2_22/pi^2,digits)}$, defining $\mathrm{E}_1^{(2)}=\sage{Upper(CONSTANT2,digits)}$.
\end{remark}

\begin{lemma}\label{seekfor}  Let $X>0$ and $q\in\mathbb{Z}_{>0}$. Then $\sum_{\substack{\ell\leq X\\(\ell,q)}}\frac{\mu^2(\ell)}{\ell}$ equals
\begin{align*}
\frac{q}{\kappa(q)}\frac{6}{\pi^2}\left(\log(X)+\mathfrak{b}_q\right)+O^*\left(\frac{\sqrt{q}}{\varphi_{\frac{1}{2}}(q)}\frac{\mathrm{E}_1^{(1)}\prod_{2|q}\frac{\mathrm{E}_1^{(2)}}{\mathrm{E}_1^{(1)}}}{\sqrt{X}}\right), 
\end{align*}
where $\mathfrak{b}_q$ is defined in Lemma \ref{corollary} and, if $\alpha>\frac{1}{2}$, $\alpha\neq 1$, $\sum_{\substack{\ell\leq X\\(\ell,q)=1}}\frac{\mu^2(\ell)}{\ell^\alpha}$ equals
\begin{align*}
\frac{q^\alpha}{\kappa_{\alpha}(q)}\frac{\zeta(\alpha)}{\zeta(2\alpha)}-\frac{q}{\kappa(q)}\frac{6}{(\alpha-1)\pi^2}\frac{1}{X^{\alpha-1}}+O^*\left(\frac{\sqrt{q}}{\varphi_{\frac{1}{2}}(q)}\frac{\mathrm{E}_\alpha^{(1)}\prod_{2|q}\frac{\mathrm{E}_\alpha^{(2)}}{\mathrm{E}_\alpha^{(1)}}}{X^{\alpha-\frac{1}{2}}}\right),
\end{align*}
 where $\mathrm{E}_\alpha^{(v)}, v\in\{1,2\}$, is defined as in Lemma \ref{sum2:critic}.
\end{lemma}
\begin{proof} Proceed as in \cite[Lemma 2.17]{MV07}. Define $\mathcal{D}_{r}=\{p\text{ prime }, p|d\implies p|r\}\subset\mathbb{Z}_{\geq 0}$. Consider $v\in\{1,2\}$ and write $q=v^kr, k\in\mathbb{Z}_{>0}$, with $(v,r)=1$ (where, if $v=1$, then $k=0$). Then for all $s\in\mathbb{C}$ such that $\Re(s)>1-\alpha$, we have the identity 
\begin{align*}
\sum_{\substack{\ell\\(\ell,q)=1}}\frac{\mu^2(\ell)}{\ell^{s+\alpha}}=\prod_{p|r}\left(1+\frac{1}{p^{s+\alpha}}\right)^{-1}\cdot\sum_{\substack{\ell\\(\ell,v)=1}}\frac{\mu^2(\ell)}{\ell^{s+\alpha}}=\sum_{\substack{d\\d\in\mathcal{D}_r}}\frac{\lambda(d)}{d^{s+\alpha}}\cdot\sum_{\substack{e\\(e,v)=1}}\frac{\mu^2(e)}{e^{s+\alpha}},
\end{align*}
where $\lambda$ corresponds to the Liouville function:  the completely multiplicative function taking the value $-1$ at every prime number.
Hence
\begin{equation}\label{nice}
\sum_{\substack{\ell\leq X\\(\ell,q)=1}}\frac{\mu^2(\ell)}{\ell^\alpha}=\sum_{\substack{d\\d\in\mathcal{D}_r}}\frac{\lambda(d)}{d^\alpha}\sum_{\substack{e\leq\frac{X}{d}\\(e,v)=1}}\frac{\mu^2(e)}{e^\alpha},
\end{equation}
which, as in Lemma \ref{SumEstimations}, does not require the condition $\{d\leq X\}$. We are considering thus an infinite range of values of $d$ for the above outer sum, which can be estimated as long as the inner sum is expressed asymptotically with an error term valid even when it has an empty condition plus the fact that the series of error terms for this expression, formed by the outer sum, converges.

If $\alpha=1$, by using Lemma \ref{sum2:critic:1} in \eqref{nice}, we derive the same main term as the one given in Corollary \ref{corollary} $\mathbf{(b)}$, but a better error term magnitude, since $\sum_{\substack{\ell\leq X\\(\ell,q)=1}}\frac{\mu^2(\ell)}{\ell}$ can be written as 
\begin{align*}
&\sum_{\substack{d\\d\in\mathcal{D}_r}}\frac{\lambda(d)}{d}\left(\frac{6}{\pi^2}\frac{v}{\kappa(v)}\left(\log\left(\frac{X}{d}\right)+\mathfrak{b}_v\right)+O^*\left(\frac{\sqrt{v}}{\varphi_{\frac{1}{2}}(v)}\frac{\mathrm{E}_{1} ^{(v)}\sqrt{d}}{\sqrt{X}}\right)\right)\\
&=\frac{vr}{\kappa(vr)}\frac{6}{\pi^2}\left(\log(X)+\mathfrak{b}_v\right)-\frac{v}{\kappa(v)}\frac{6}{\pi^2}\sum_{\substack{d\\d\in\mathcal{D}_r}}\frac{\lambda(d)\log(d)}{d}\\
&\phantom{xxxxxxxxxxxxxxxxxxll}+O^*\left(\frac{\sqrt{v}}{\varphi_{\frac{1}{2}}(v)}\sum_{\substack{d\\d\in\mathcal{D}_r}}\frac{\mathrm{E}_{1} ^{(v)}}{\sqrt{d}}\cdot\frac{1}{\sqrt{X}}\right)\\
&=\frac{q}{\kappa(q)}\frac{6}{\pi^2}\left(\log(X)+\mathfrak{b}_q\right)+O^*\left(\frac{\sqrt{q}}{\varphi_{\frac{1}{2}}(q)}\frac{\mathrm{E}_1^{(1)}\prod_{2|q}\frac{\mathrm{E}_1^{(2)}}{\mathrm{E}_1^{(1)}}}{\sqrt{X}}\right),
\end{align*}
where we have used that
\begin{align*}
\sum_{\substack{d\\d\in\mathcal{D}_r}}\frac{-\lambda(d)\log(d)}{d}=\frac{r}{\kappa(r)}\left(\sum_{\substack{d\\d\in\mathcal{D}_r}}\frac{\lambda(d)}{d^s}\right)^{-1}_{s=1}\cdot\left(\sum_{\substack{d\\d\in\mathcal{D}_r}}\frac{\lambda(d)}{d^s}\right)'_{s=1}\phantom{xxxxxxx}&\\
=\frac{r}{\kappa(r)}\sum_{p|r}\left[\left(\left(1+\frac{1}{p^s}\right)^{-1}\right)'\left(1+\frac{1}{p^s}\right)\right]_{s=1}=\frac{r}{\kappa(r)}\sum_{p|r}\frac{\log(p)}{p+1},&
\end{align*}
and that $\frac{vr}{\kappa(vr)}=\frac{q}{\kappa(q)}$, $\frac{\sqrt{vr}}{\varphi_{\frac{1}{2}}(vr)}=\frac{\sqrt{q}}{\varphi_{\frac{1}{2}}(q)}$, $\sum_{p|v}\frac{\log(p)}{p+1}+\sum_{p|r}\frac{\log(p)}{p+1}=\sum_{p|q}\frac{\log(p)}{p+1}$.

Finally, if $\alpha\neq 1$, then by using Lemma \ref{sum2:critic} in \eqref{nice} and by noticing that $\frac{(vr)^\alpha}{\kappa_\alpha(vr)}=\frac{q^\alpha}{\kappa_\alpha(q)}$, we derive that $\sum_{\substack{\ell\leq X\\(\ell,q)=1}}\frac{\mu^2(\ell)}{\ell^\alpha}$ can be expressed as 
\begin{align*}
&\sum_{\substack{d\\d\in\mathcal{D}_r}}\frac{\lambda(d)}{d^\alpha}\left(\frac{v^\alpha}{\kappa_\alpha(v)}\frac{\zeta(\alpha)}{\zeta(2\alpha)}-\frac{v}{\kappa(v)}\frac{6}{(\alpha-1)\pi^2}\frac{d^{\alpha-1}}{X^{\alpha-1}}+O^*\left(\frac{\sqrt{v}}{\varphi_{\frac{1}{2}}(v)}\frac{\mathrm{E}_\alpha^{(v)} d^{\alpha-\frac{1}{2}}}{X^{\alpha-\frac{1}{2}}}\right)\right)\\
&=\frac{q^\alpha}{\kappa_{\alpha}(q)}\frac{\zeta(\alpha)}{\zeta(2\alpha)}-\frac{q}{\kappa(q)}\frac{6}{(\alpha-1)\pi^2}\frac{1}{X^{\alpha-1}}+O^*\left(\frac{\sqrt{q}}{\varphi_{\frac{1}{2}}(q)}\frac{\mathrm{E}_\alpha^{(1)}\prod_{2|q}\frac{\mathrm{E}_\alpha^{(2)}}{\mathrm{E}_\alpha^{(1)}}}{X^{\alpha-\frac{1}{2}}}\right),
\end{align*}
which, again, has the expected main term according to Theorem \ref{general} but an error term of lower magnitude.
\end{proof}

Let us recall that the requirement of the empty sum estimation, as in Lemma \ref{SumEstimations}, worsens a bit the error term constants with respect to the ones under condition $X\geq 1$, say, as shown in lemmas \ref{sum2:critic:1} and \ref{sum2:critic}, but we gain regularity in our expressions in the variable $d$. It is precisely that regularity that allows us to derive the coprimality restrictions products in a simpler manner: for example, we derive immediately that $\sum_{\substack{d\\d\in\mathcal{D}_r}}\frac{\lambda(d)}{d}=\frac{r}{\kappa(r)}$, whereas condition $\frac{X}{d}\geq 1$ would have imposed us to analyze $\sum_{\substack{d\leq X\\d\in\mathcal{D}_r}}\frac{\lambda(d)}{d}$ or, rather, $\sum_{\substack{d>X\\d\in\mathcal{D}_r}}\frac{\lambda(d)}{d}$. This last observation is key for the work carried out in \cite{RA13} and \cite{RA19}. 

\begin{corollary} Let $X>0$. Then
\begin{align*}
\sum_{\substack{\ell> X\\(\ell,q)=1}}\frac{\mu^2(\ell)}{\ell^2}=\frac{q}{\kappa(q)}\frac{6}{\pi^2}\frac{1}{X}+O^*\left(\frac{\sqrt{q}}{\varphi_{\frac{1}{2}}(q)}\frac{\sage{Upper(E(2,1),digits)}}{X^{\frac{3}{2}}}\right),&\quad\text{ if }2\nmid q,\\
\phantom{\sum_{\substack{\ell> X\\(\ell,q)=1}}\frac{\mu^2(\ell)}{\ell^2}}=\frac{q}{\kappa(q)}\frac{6}{\pi^2}\frac{1}{X}+O^*\left(\frac{\sqrt{q}}{\varphi_{\frac{1}{2}}(q)}\frac{\sage{Upper(E(2,2),digits)}}{X^{\frac{3}{2}}}\right),&\quad\text{ if }2|q.
\end{align*} 
\begin{proof} By applying Lemma \ref{seekfor} with $\alpha=2$, we have
\begin{align*}
\sum_{\substack{\ell\leq X\\(\ell,q)=1}}\frac{\mu^2(\ell)}{\ell^2}&=\frac{q^2}{\kappa_{2}(q)}\frac{\zeta(2)}{\zeta(4)}-\frac{q}{\kappa(q)}\frac{6}{\pi^2}\frac{1}{X}+O^*\left(\frac{\sqrt{q}}{\varphi_{\frac{1}{2}}(q)}\frac{\mathrm{E}_2^{(1)}\prod_{2|q}\frac{\mathrm{E}_2^{(2)}}{\mathrm{E}_2^{(1)}}}{X^{\frac{3}{2}}}\right),
\end{align*}
where, for $v\in\{1,2\}$, $\mathrm{E}_2^{(v)}$ is defined as
\begin{align*}
\max\left\{\frac{5\ \mathrm{D}_v}{3},\frac{\varphi_{\frac{1}{2}}(v)}{\sqrt{v}}\left|\frac{v^2}{\kappa_2(v)}\frac{\zeta(2)}{\zeta(4)}-\frac{v}{\kappa(v)}\frac{6}{\pi^2}\right|,\frac{\varphi_{\frac{1}{2}}(v)}{\sqrt{v}}\frac{2}{3}\left(\frac{2\kappa_2(v)\zeta(4)}{v\kappa(v)\pi^2\zeta(2)}\right)^{2}\right\}&\\
\leq\begin{cases}
\sage{Upper(E(2,1),digits)},&\quad\text{ if }v=1,\\
\sage{Upper(E(2,2),digits)},&\quad\text{ if }v=2.
\end{cases}&
\end{align*}
We obtain the result by observing that
\begin{equation*}
\sum_{\substack{\ell> X\\(\ell,q)=1}}\frac{\mu^2(\ell)}{\ell^2}=\frac{q^2}{\kappa_{2}(q)}\frac{\zeta(2)}{\zeta(4)}
-\sum_{\substack{\ell\leq X\\(\ell,q)=1}}\frac{\mu^2(\ell)}{\ell^2}.
\end{equation*} 
\end{proof} 
\end{corollary}

\subsection{Achieving the critical exponent}\label{achieving}

We present a new method to achieve the critical exponent for estimations of averages of the form studied in Theorem \ref{general} provided that the difference between $\beta$ and $\alpha$ defined therein is strictly bigger than $\frac{1}{2}$: in this case, we go to the edge of the special form of the convolution method given in \S\ref{TCM}; moreover, no extra conditions on $\beta$ are needed but $\beta-\alpha>\frac{1}{2}$. Nonetheless, if $\beta-\alpha\leq\frac{1}{2}$, then we should still refer to Theorem \ref{general} and its choice of parameter (or indirectly to it, as shown by summation by parts in Theorem \ref{general++} $\mathbf{(B)}$, $\mathbf{(C)}$).   

\begin{theorem}\label{general++}  
  Let $X>0$, be a real number and $q$ a positive integer. Consider a multiplicative function $f:\mathbb{Z}^+\to\mathbb{C}$ such that for every prime number $p$ satisfying $(p,q)=1$, we have $f(p)=\frac{1}{p^{\alpha}}+O\left(\frac{1}{p^{\beta}}\right)$, where $\alpha$, $\beta$ are real numbers satisfying $\beta>\alpha$, $\beta-\alpha>\frac{1}{2}$. We have the following

 \noindent $\mathbf{(A)}$ If $\alpha>\frac{1}{2}$ then 
 \begin{align*}
 \sum_{\substack{\ell\leq X\\ (\ell,q)=1}}\mu^2(\ell)f({\ell})=F_\alpha^{q}(X)+O^*\left(\mathrm{p}_\alpha(q)\cdot\frac{\mathrm{w}_\alpha^q\ \mathrm{P}_\alpha}{X^{\alpha-\frac{1}{2}}}\right),
 \end{align*} 
 where $F_\alpha^q(X)$ is defined as in Theorem \ref{general}, and, if $2|q$, $\mathrm{w}_{\alpha}^q=\mathrm{E}_\alpha^{(2)}$, whereas, if $2\nmid q$, 
\begin{align*}
\mathrm{w}_\alpha^q&=\left(\frac{\sqrt{2}-1}{\sqrt{2}-1+|2^\alpha f(2)-1|}\right)\left(\mathrm{E}_\alpha^{(1)}+\frac{|2^\alpha f(2)-1|\ \mathrm{E}_\alpha^{(2)}}{\varphi_{\frac{1}{2}}(2)}\right).
\end{align*}
Here $\mathrm{E}_\alpha^{(v)}, v\in\{1,2\}$ is defined in Lemma \ref{sum2:critic} and Remark \ref{newdef}, and we have 
\begin{align*}
\mathrm{p}_\alpha(q)=\prod_{p|q}\left(1+\frac{1-|f(p)p^\alpha-1|}{\sqrt{p}-1+|f(p)p^{\alpha}-1|}\right),\mathrm{P}_\alpha=\prod_{p}\left(1+\frac{|f(p)p^\alpha-1|}{\sqrt{p}-1}\right),
 \end{align*}    
for all $\alpha$, where $\mathrm{P}_\alpha$ is a convergent infinite product. 

 \noindent $\mathbf{(B)}$ If $\alpha<\frac{1}{2}$ then $ \sum_{\substack{\ell\leq X\\ (\ell,q)=1}}\mu^2(\ell)f({\ell})$ can be expressed as %%% Summation by parts with 1^- = 0 %%%
 \begin{equation*}
\frac{H_{f'}^q(0)\varphi(q)}{(1-\alpha)q}X^{1-\alpha}+O^*\left(\mathrm{p}_{\alpha}(q)\cdot\left(1+\frac{2-2\alpha}{1-2\alpha}\right)\mathrm{w'}_{\alpha}^q\mathrm{P}_{\alpha} \ X^{\frac{1}{2}-\alpha}\right),
 \end{equation*} 
where $\mathrm{p}_{\alpha}(q)$ and $\mathrm{P}_{\alpha}$ are as in $\mathbf{(A)}$ and for $\alpha\leq\frac{1}{2}$,
 \begin{align*}
H_{f'}^q(0)&=\prod_{p\nmid q}\left(1-\frac{p^{1-\alpha}-f(p)p+f(p)}{p^{2-\alpha}} \right),\\
\mathrm{w'}_{\alpha}^q&= 
\begin{cases}\mathrm{E}_1^{(2)}=\sage{Upper(CONSTANT2,digits)},&\quad\text{ if }2|q,\\
\left(\frac{\sqrt{2}-1}{\sqrt{2}-1+|2^\alpha f(2)-1|}\right)\left(\mathrm{E}_1^{(1)}+\frac{|2^\alpha f(2)-1|\ \mathrm{E}_1^{(2)}}{\varphi_{\frac{1}{2}}(2)}\right),&\quad\text{ if }2\nmid q.\\
\end{cases}
\end{align*} 

 \noindent $\mathbf{(C)}$ If $\alpha=\frac{1}{2}$ then $\sum_{\substack{\ell\leq X\\ (\ell,q)=1}}\mu^2(\ell)f({\ell})$ can be written as %%% Summation by parts with 1^- ---> 1 %%%
 \begin{equation*}
\frac{H_{f'}^q(0)\varphi(q)}{(1-\alpha)q}X^{1-\alpha}+O^*\left(\mathrm{C}+\mathrm{p}_{\alpha}(q)\mathrm{w'}_{\alpha}^q\mathrm{P}_{\alpha} \ \left(1+\frac{1}{2}\log(X)\right)\right),
 \end{equation*} 
where $\mathrm{p}_{\alpha}(q)$ and $\mathrm{P}_{\alpha}$ are as in $\mathbf{(A)}$, $H_{f'}^q(0)$ and $\mathrm{w'}_{\alpha}^q$ are as in $\mathbf{(B)}$ and
\begin{align*}
\mathrm{C}&=\left|\frac{H_{f'}^{q}(0)\varphi(q)}{q}\left(\sum_{p\nmid q}\frac{\log(p)(\sqrt{p}-(p-2)f(p))}{(f(p)+\sqrt{p})(p-1)}+\gamma+\sum_{p|q}\frac{\log(p)}{p-1}-2\right)\right|.
\end{align*}
 \end{theorem}

\begin{proof} Let us derive $\mathbf{(A)}$. Consider the arithmetic function $i_f$ defined on each prime as $p\mapsto f(p)p^\alpha-1$. Observe that  
\begin{align}
 \sum_{\substack{\ell\leq X\\ (\ell,q)=1}}\mu^2(\ell)f({\ell})=\sum_{\substack{\ell\leq X\\ (\ell,q)=1}}\frac{\mu^2(\ell)}{\ell^\alpha}\cdot f(\ell)\ell^{\alpha}=\sum_{\substack{\ell\leq X\\ (\ell,q)=1}}\frac{\mu^2(\ell)}{\ell^\alpha}\cdot \prod_{p|\ell}(1+i_f(p))&\nonumber\\
=\sum_{\substack{\ell\leq X\\ (\ell,q)=1}}\frac{\mu^2(\ell)}{\ell^\alpha}\sum_{d|\ell}\mu^2(d)i_f(d)=\sum_{\substack{d\\(d,q)=1}}\frac{\mu^2(d)i_f(d)}{d^{\alpha}}\sum_{\substack{e\leq\frac{X}{d}\\(e,qd)=1}}\frac{\mu^2(e)}{e^\alpha}&\label{followup},
 \end{align} 
where we have not imposed upper bound conditions on the variable $d$. 

In order to continue our estimation, we must be able to estimate the innermost summation in \eqref{followup} regardless of whether or not it has an empty condition, so that their remainder terms converge upon effecting the corresponding outermost summation. As $\alpha>\frac{1}{2}$, this situation can be treated with the help of Lemma \ref{seekfor}; we distinguish two cases.

\noindent $\mathbf{i)}$ $2|q$. Then continuing from \eqref{followup}, along with the ideas of the proof of Theorem \ref{general} and Lemma \ref{seekfor}, it is not difficult to see, as expected, that for all $\alpha>\frac{1}{2}$, the main term of $\sum_{\substack{\ell\leq X\\ (\ell,q)=1}}\mu^2(\ell)f({\ell})$ is $F_\alpha^q(X)$. As for the error term, it corresponds to
\begin{align*}\sum_{\substack{d\\(d,q)=1}}\frac{\mu^2(d)|i_f(d)|}{d^\alpha}O^*\left(\frac{\sqrt{qd}}{\varphi_{\frac{1}{2}}(qd)}\frac{\mathrm{E}_\alpha^{(2)}\ d^{\alpha-\frac{1}{2}}}{X^{\alpha-\frac{1}{2}}}\right)\phantom{xxxxxxxxx}&\\
=O^*\left(\frac{\sqrt{q}}{\varphi_{\frac{1}{2}}(q)}\prod_{p\nmid q}\left(1+\frac{|i_f(p)|}{\sqrt{p}-1}\right)\cdot\frac{\mathrm{E}_\alpha^{(2)}}{X^{\alpha-\frac{1}{2}}}\right),&
\end{align*}
where, for any $\alpha>\frac{1}{2}$, $\frac{\sqrt{q}}{\varphi_{\frac{1}{2}}(q)}\prod_{p\nmid q}\left(1+\frac{|i_f(p)|}{\sqrt{p}-1}\right)$ may be expressed as
\begin{align*}
\frac{\sqrt{q}}{\varphi_{\frac{1}{2}}(q)}\prod_{p|q}\left(1+\frac{|f(p)p^\alpha-1|}{\sqrt{p}-1}\right)^{-1}\cdot\mathrm{P}_\alpha=\mathrm{p}_\alpha(q)\cdot\mathrm{P}_\alpha.
\end{align*}
where $\mathrm{p}_\alpha(q)$ and  $\mathrm{P}_\alpha$ are defined in the statement. Observe that $\mathrm{P}_\alpha$ converges, as $\frac{|i_f(p)|}{\sqrt{p}-1}=\frac{|f(p)p^\alpha-1|}{\sqrt{p}-1}=O\left(\frac{1}{p^{\beta-\alpha+\frac{1}{2}}}\right)$ and $\beta-\alpha+\frac{1}{2}>1$.

\noindent $\mathbf{ii)}$ $2\nmid q$. Then we can write \eqref{followup} as
\begin{align*}
 &\sum_{\substack{d\\(d,2q)=1}}\frac{\mu^2(d)i_f(d)}{d^{\alpha}}\sum_{\substack{e\leq\frac{X}{d}\\(e,qd)=1}}\frac{\mu^2(e)}{e^\alpha}+\frac{i_f(2)}{2^\alpha}\sum_{\substack{d\\(d,2q)=1}}\frac{\mu^2(d)i_f(d)}{d^{\alpha}}\sum_{\substack{e\leq\frac{X}{2d}\\(e,2qd)=1}}\frac{\mu^2(e)}{e^\alpha}\\
&=S_\alpha^q(X)+\frac{i_f(2)}{2^\alpha}T_\alpha^q(X).
\end{align*}
Again, it is not difficult to see that, for any $\alpha>\frac{1}{2}$, the main term of $S_\alpha^q(X)+\frac{i_f(2)}{2^\alpha}T_\alpha^q(X)$ is $F_\alpha^q(X)$, defined in Theorem \ref{general}. On the other hand, the error term of $S_1^q(X)+\frac{i_f(2)}{2}T_1^q(X)$, it can be expressed as
\begin{align*}
 &\sum_{\substack{d\\(d,2q)=1}}\frac{\mu^2(d)|i_f(d|)}{d^{\alpha}}O^*\left(\frac{\sqrt{qd}}{\varphi_{\frac{1}{2}}(qd)}\frac{\mathrm{E}_{\alpha}^{(1)}\ d^{\alpha-\frac{1}{2}}}{X^{\alpha-\frac{1}{2}}}\right)\\
&\phantom{xxxx}+\frac{|i_f(2)|}{2^{\alpha}}\sum_{\substack{d\\(d,2q)=1}}\frac{\mu^2(d)|i_f(d)|}{d^{\alpha}}O^*\left(\frac{\sqrt{2qd}}{\varphi_{\frac{1}{2}}(2qd)}\frac{\mathrm{E}_{\alpha}^{(2)}\ (2d)^{\alpha-\frac{1}{2}}}{X^{\alpha-\frac{1}{2}}}\right)=\\
&O^*\left(\frac{\sqrt{q}}{\varphi_{\frac{1}{2}}(q)}\prod_{p\nmid 2q}\left(1+\frac{|i_f(p)|}{\sqrt{p}-1}\right)\left(\mathrm{E}_\alpha^{(1)}+\frac{|i_f(2)|\ \mathrm{E}_\alpha^{(2)}}{\varphi_{\frac{1}{2}}(2)}\right)\cdot\frac{1}{X^{\alpha-\frac{1}{2}}}\right)=\\
&O^*\left(\mathrm{p}_\alpha(q)\left(\frac{\sqrt{2}-1}{\sqrt{2}-1+|2^\alpha f(2)-1|}\right)\left(\mathrm{E}_\alpha^{(1)}+\frac{|2^\alpha f(2)-1|\ \mathrm{E}_\alpha^{(2)}}{\varphi_{\frac{1}{2}}(2)}\right)\cdot\frac{\mathrm{P}_\alpha}{X^{\alpha-\frac{1}{2}}}\right),
\end{align*}
whence the first case.

Condition $\alpha>\frac{1}{2}$ in the case $\mathbf{(A)}$ is necessary, as we have used Lemma \ref{seekfor}. Nonetheless, we can readily derive an analogous result for the cases $\mathbf{(B)}$  and $\mathbf{(C)}$. Indeed, we can write $f(p)=p^{1-\alpha}f'(p)$, where $A(t)=\sum_{\substack{\ell\leq t\\(\ell,q)=1}}\mu^2(\ell)f'(\ell)$ can be estimated by the case $\mathbf{(A)}$ with $\alpha'=1$, $\beta'=1-\alpha+\beta$. We can then estimate $\sum_{\substack{\ell\leq X\\(\ell,q)=1}}\mu^2(\ell)f(\ell)=\sum_{\substack{\ell\leq X\\(\ell,q)=1}}\mu^2(\ell)f'(\ell)\ell^{1-\alpha}$ by means of a summation by parts, obtaining the result.
\end{proof}

Note that the error term improvement from Theorem \ref{general}, when $\alpha=\frac{1}{2}$ and under conditions of Theorem \ref{general++}, is of logarithmic nature with respect to $O(X^{\frac{1}{2}-\delta})$ for any $\delta\in(0,\frac{1}{2})$.

Concerning the error term in Theorem \ref{general++}, in some particular cases one can do much better in terms of error constants. For instance, it is known, by  \cite[Lemmas 5.1-5.2]{Hel19}
that if $f(p)=1$ and $v\in\{1,2\}$, we have that for any $X>0$ that
\begin{equation}\label{squarefree}  
\sum_{\substack{\ell\leq X\\(\ell,v)=1}}\mu^2(\ell)=\frac{6}{\pi^2}\frac{v}{\kappa(v)}X+O^*(\mathrm{H}_v\sqrt{X}),
\end{equation}
where
\begin{equation}\mathrm{H}_v=
\begin{cases}
\sqrt{3}\left(1-\frac{6}{\pi^2}\right)&\quad\text{ if }v=1,\label{HH2}\\
1-\frac{4}{\pi^2}&\quad\text{ if }v=2,
\end{cases}
\end{equation}
whereas Corollary \ref{general++}  provides only an explicit error term of the form $O^*\left(\frac{\sqrt{q}}{\varphi_{\frac{1}{2}}(q)}\cdot\sage{Upper(cst_crucial*3,digits)}\sqrt{X}\right)$.

\subsection{Consequences}\label{Cop}

\begin{lemma}\label{consequences}
Let $X>0$, then the sum $\sum_{\substack{\ell\leq X\\(\ell,q)=1}}\frac{\mu^2(\ell)}{\varphi(\ell)}$ may be estimated as
\begin{equation}\label{summi} 
\frac{\varphi(q)}{q}\left(\log\left(X\right)+\mathfrak{a}_q\right)+O^*\left(\prod_{p|q}\left(1+\frac{p-2}{p^{\frac{3}{2}}-p-\sqrt{p}+2}\right)\cdot\frac{\sage{Upper(CONSTANT_RAM*Prod_Ram_Upper,digits)}\prod_{2|q}\sage{Upper(CONSTANT2/CONSTANT_RAM,digits)}}{\sqrt{X}}\right),
\end{equation}
where $\mathfrak{a}_q$ is defined in Corollary \ref{corollary}.
\end{lemma}
\begin{proof} We already know the main term of the asymptotic expression of the above sum, thanks to Corollary \ref{corollary} $\mathbf{(a)}$; obtaining it again from Theorem \ref{general++} is an exercise. On the other hand, by Theorem \ref{general++} with $f(p)=\frac{1}{p-1}$, $\alpha=1$, $\beta=2$, its error term can be expressed as $O^*\left(\mathrm{p}(q)\cdot\frac{\mathrm{w}^q\  \mathrm{P}}{\sqrt{X}}\right)$, where
\begin{align*}
\mathrm{p}(q)&=\prod_{p|q}\left(1+\frac{p-2}{p^{\frac{3}{2}}-p-\sqrt{p}+2}\right),\\
 \mathrm{P}&=\prod_{p}\left(1+\frac{1}{(p-1)(\sqrt{p}-1)}\right)\in[\sage{Trunc(Prod_Ram_Lower,dlong)},\sage{Trunc(Prod_Ram_Upper,dlong)}],\\
\mathrm{w}^q&=\begin{cases}
\sage{Trunc(CONSTANT2,digits)},&\text{ if }2|q,\\
\left(1-\frac{1}{\sqrt{2}}\right)\left(\mathrm{E}_1^{(1)}+\frac{\mathrm{E}_1^{(2)}}{\varphi_{\frac{1}{2}}(2)}\right)=\sage{Trunc(CONSTANT_RAM,digits)}\ldots,&\text{ if }2\nmid q
\end{cases}\leq\sage{Upper(CONSTANT_RAM,digits)}\prod_{2|q}\sage{Upper(CONSTANT2/CONSTANT_RAM,digits)},
\end{align*} 
and where $\mathrm{E}_1^{(v)}$, $v\in\{1,2\},$ is defined in \S\ref{particular}.
\end{proof}

When there is no coprimality conditions, we have obtained an error constant equal to $\sage{Upper(CONSTANT_RAM*Prod_Ram_Upper,digits)}$, that held under condition $X>0$. Ramar\'e and Akhilesh in \cite[Thm. 1.2]{RA13} have given the constant $3.95$ under the condition $X\geq 1$, later improved by Ramar\'e himself in \cite{RA19} to $2.44$ under the condition $X>1$. From these last two bounds, it is not difficult to extend the range of estimation to $X>0$, as we have done for example throughout Lemma \ref{SumEstimations}, and these bounds continue to be better than the value $\sage{Upper(CONSTANT_RAM*Prod_Ram_Upper,digits)}$.

Nonetheless, the above lemma improve considerably \cite[Thm. 1.1]{RA13} when coprimality conditions given by $q\geq 2$ are involved. For example, we have
\begin{align}\label{values}
\sage{Upper(crux(v0)*Prod_Ram_Upper,digits)}\cdot\mathrm{p}(\sage{v0})\leq\sage{Upper(pp(v0)*crux(v0)*Prod_Ram_Upper,digits)}&\leq\sage{Lower(5.9*ramp(v0),digits)}\leq 5.9\cdot j(\sage{v0}),\nonumber\\
\sage{Upper(crux(v1)*Prod_Ram_Upper,digits)}\cdot\mathrm{p}(\sage{v1})\leq\sage{Upper(pp(v1)*crux(v1)*Prod_Ram_Upper,digits)}&\leq\sage{Lower(5.9*ramp(v1),digits)}\leq 5.9\cdot j(\sage{v1}),\nonumber\\
\sage{Upper(crux(v2)*Prod_Ram_Upper,digits)}\cdot\mathrm{p}(\sage{v2})\leq\sage{Upper(pp(v2)*crux(v2)*Prod_Ram_Upper,digits)}&\leq\sage{Lower(5.9*ramp(v2),digits)}\leq 5.9\cdot j(\sage{v2}),\\
\sage{Upper(crux(v1_2)*Prod_Ram_Upper,digits)}\cdot\mathrm{p}(\sage{v1_2})\leq\sage{Upper(pp(v1_2)*crux(v1_2)*Prod_Ram_Upper,digits)}&\leq\sage{Lower(5.9*ramp(v1_2),digits)}\leq 5.9\cdot j(\sage{v1_2}),\nonumber\\
\sage{Upper(crux(v3_2)*Prod_Ram_Upper,digits)}\cdot\mathrm{p}(\sage{v3_2})\leq\sage{Upper(pp(v3_2)*crux(v3_2)*Prod_Ram_Upper,digits)}&\leq\sage{Lower(5.9*ramp(v3_2),digits)}\leq 5.9\cdot j(\sage{v3_2}),\nonumber\\
\sage{Upper(crux(v7_2)*Prod_Ram_Upper,digits)}\cdot\mathrm{p}(\sage{v7_2})\leq\sage{Upper(pp(v7_2)*crux(v7_2)*Prod_Ram_Upper,digits)}&\leq\sage{Lower(5.9*ramp(v7_2),digits)}\leq 5.9\cdot j(\sage{v7_2}),\nonumber
\end{align}
where $j$ is the error term arithmetic function defined in \cite[Thm. 1.1]{RA13} as $2\mapsto\frac{21}{25}$ and $p\geq 3\mapsto 1+\frac{p-2}{p^{\frac{3}{2}}-\sqrt{p}+1}$. Furthermore, the estimation given in Lemma \ref{consequences} is better than the one in \cite[Thm. 1.1]{RA13} for all $q=p$ prime. Indeed, we observe in \eqref{values} that it is better when $p\in\{2,3,5\}$; now, since
\begin{align*}
\frac{p-2}{p^{\frac{3}{2}}-p-\sqrt{p}+2}&<\frac{1}{\sqrt{p}}&&\text{ for all }p\geq 3,\\
\frac{p-2}{p^{\frac{3}{2}}-\sqrt{p}+1}&>\frac{1}{2\sqrt{p}}&&\text{ for all }p\geq 5,
\end{align*}
we have, for all $p\geq 3$, that
\begin{align*}
\sage{Upper(crux(v1)*Prod_Ram_Upper,digits)}\cdot\mathrm{p}(p)\leq\sage{Upper(crux(v1)*Prod_Ram_Upper,digits)}\cdot\left(1+\frac{1}{\sqrt{p}}\right)\leq5.9\cdot\left(1+\frac{1}{2\sqrt{p}}\right)\leq 5.9\cdot j(p),
\end{align*}
whence the conclusion.

As a final remark, observe that that the main contribution to the product $\mathrm{P}$ given in Lemma \ref{consequences} is precisely when $p=2$. This is the reason why, in the present work, we have distinguished if $q$ is either odd or even. Further, as the second main contribution to the product $\mathrm{P}$ is given by its factor at $p=3$ (the subsequent factors when $p>3$ being rather small, as $\frac{1}{\sqrt{p}-1}<1$), the interested reader may study the behavior of the error term bounds given in Theorem \ref{general++}, and therefore the error term in Lemma \ref{consequences}, by distinguishing whether or not $(6,q)=1$: this procedure will require an extension of Lemma \ref{sum2:critic:1} to the cases $(3,q)=1$ and, by using the inclusion-exclusion principle, to the case $(6,q)=1$; afterwards, the analysis will continue exactly as in the current version of Theorem \ref{general++}.

\end{document}